\documentclass[english,11pt]{article}
\usepackage[T1]{fontenc}
\usepackage{babel}
\usepackage{amsmath,amsfonts,amsthm}
\usepackage{color}
\usepackage{pdfsync}
\usepackage{graphicx}
\usepackage{url}

\DeclareMathSymbol{\leqslant}{\mathalpha}{AMSa}{"36} 
\DeclareMathSymbol{\geqslant}{\mathalpha}{AMSa}{"3E} 
\renewcommand{\leq}{\;\leqslant\;}                   
\renewcommand{\geq}{\;\geqslant\;}                   

\newtheorem{Th}{Theorem}
\newtheorem{Le}[Th]{Lemma}
\newtheorem{Pro}[Th]{Proposition}
\newtheorem{Def}[Th]{Definition}
\newtheorem{Cor}[Th]{Corollary}

\newcommand{\cA}{\ensuremath{\mathcal A}}

\newcommand{\cD}{\ensuremath{\mathcal D}}
\newcommand{\cE}{\ensuremath{\mathcal E}}
\newcommand{\cF}{\ensuremath{\mathcal F}}

\newcommand{\cH}{\ensuremath{\mathcal H}}

\newcommand{\cO}{\ensuremath{\mathcal O}}
\newcommand{\cP}{\ensuremath{\mathcal P}}

\newcommand{\cR}{\ensuremath{\mathcal R}}

\newcommand{\cX}{\ensuremath{\mathcal X}}

\newcommand{\bbD}{{\ensuremath{\mathbb D}} }
\newcommand{\bbE}{{\ensuremath{\mathbb E}} }

\newcommand{\bbN}{{\ensuremath{\mathbb N}} }

\newcommand{\bbP}{{\ensuremath{\mathbb P}} }

\newcommand{\bbR}{{\ensuremath{\mathbb R}} }

\newcommand{\Om}{\Omega}
\newcommand{\E}{\bbE}

\newcommand{\N}{\bbN}
\newcommand{\Ne}{\bbN^{\ast}}
\newcommand{\R}{\bbR}
\newcommand{\bbd}{\mathbf{d}}

\newcommand{\bbbd}{\bar{\bbd}}
\newcommand{\bbbD}{\bar{\bbD}}

\newcommand{\tN}{\tilde{N}}
\newcommand{\hOm}{\hat{\Om}}
\newcommand{\hcA}{\hat{\cA}}
\newcommand{\hbbP}{\hat{\bbP}}
\newcommand{\hE}{\hat{\E}}

\newcommand{\crea}{\varepsilon^+}
\newcommand{\anni}{\varepsilon^-}
\newcommand{\sbbD}{\underline{\bbD}}
\newcommand{\1}{\mathbf{1}}

\newcommand{\bK}{\bar{K}}

\newcommand{\LDP}[1]{{\cH}_{\bbD^{#1},\cP}}
\newcommand{\LDdP}[1]{{\cH}_{(\bbD\ho{\bbd})^{#1},\cP}}

\newcommand{\bLDdP}[1]{{\cH}_{(\widetilde{\bbD\ho\bbd })^{#1},\cP}}
\newcommand{\bLDP}[1]{{\cH}_{\bbbD^{#1},\cP}}

\newcommand{\ho}{\hat{\otimes}}

\title{Iteration of the lent particle method for\\existence of smooth densities of Poisson functionals }\date{}
\author{Nicolas BOULEAU and Laurent DENIS}
\begin{document}
\maketitle

\date

\begin{abstract}
In previous works (\cite{bouleau-denis,bouleau-denis2}) we have introduced a new method called {\it the lent particle method} which is an efficient tool to establish  
existence of densities for Poisson functionals. We now go further and iterate this method in order to prove smoothness of densities. More precisely, we construct 
Sobolev spaces of any order and prove a Malliavin-type criterion of existence of smooth density. We apply this approach to SDE's driven by Poisson random measures and also 
present some non-trivial examples to which our method applies.

\end{abstract}

\section{Introduction}
{
In order to study the regularity of the distributions of  Wiener functionals, in particular solutions of SDE's driven by a Wiener process, Dirichlet forms have shown their interest thanks to the weak hypotheses they need (cf. \cite{bouleau-hirsch2}). Soon after the extension of the Malliavin calculus to the jump case (Bismut, Bichteler, Gravereaux, Jacod, L\'eandre) in the 1980's, Agn\`es Coquio obtained existence of density results  for SDE's driven by Poisson random measures by Dirichlet forms methods (cf. \cite{coquio}) weakening, as in the continuous case, the assumptions on the coefficients of the SDE.

We took anew this approach, generalizing  the method thanks to an important simplification expressed by a formula : {\it the lent particle formula} (cf. \cite{bouleau4}, \cite{bouleau-denis}, \cite{bouleau-denis2}), which expresses the gradient on the Poisson space (upper space) in function of the gradient on the initial space (bottom space) and the creation and annihilation operators.

This representation of the gradient is expressed in functional spaces which} allow{s} easily an iteration of the gradient thanks to product spaces. 

We take here advantage of this algebraic simplicity to obtain results of existence of $\mathcal{C}^\infty$-density.  Dirichlet forms methods are rarely used for this only goal. What they bring is not, in the present redaction,  weakening  the assumptions, but clarifying and generalizing the method, with especially the possibility to address the case of an infinite dimensional initial space (bottom space) as the Wiener space itself. We illustrate this direction by an example where a particle jumps with jumps given by an auxiliary diffusion.
Finally this study shows that the language of Dirichlet forms allows easily  to extend the Malliavin calculus to Poisson measures and the introduced notions: Sobolev spaces, gradients, iterated gradients... are simple and interesting by themselves and we think they are adapted to a deeper investigation of extensions of classical functional inequalities.

{Part 2 is devoted to recalling previous results and fixing the notation. Sobolev spaces are defined in Part 3. The connexion with classical hypotheses and functional inequalities (Khintchine and Meyer inequalities) are exposed in Part 4. Next, Part 5 and 6 are devoted to the criterion of regularity and its application to solutions of SDE's and Part 7 illustrates the results by some non classical examples.}
\section{Notation and hypotheses.}
%
%
%

\subsection{Hypotheses on the bottom structure}{\label{notations}}
Let $(\Xi,\cX ,\nu,\bbd,\gamma)$ be a local symmetric Dirichlet
structure which admits a carr\'e du champ operator i.e. $(\Xi,\cX
,\nu )$ is a measured space, $\nu$ is $\sigma$-finite and the
bilinear form
\[ e [f,g]=\frac12\int\gamma [f,g]\, d\nu,\]
is a local  Dirichlet form with domain $\bbd\subset L^2 (\nu )$
and carr\'e du champ operator $\gamma$ (see Bouleau-Hirsch
\cite{bouleau-hirsch2}, Chap. I). We assume that for all $x\in \Xi$,
$\{ x\}$ belongs to $\cX$ and that $\nu$ is diffuse
($\nu(\{x\})=0\;\forall x$). The structure  $(\Xi,\cX
,\nu,\bbd,\gamma)$ is called the {\it bottom structure}. Following
\cite{bouleau-denis} \S 3.2.2, we are given an auxiliary
probability space $(R,\cR ,\rho)$ such that the dimension of the
vector space $L^2 (R,\cR
 ,\rho)$ be infinite.\\
We assume that the Hilbert space $\bbd$ is separable so that the
bottom Dirichlet structure admits a gradient operator, {denoted by $\flat$}, and we
choose a version of it with values in
 the  space $L_0=\{ g\in L^2 (R,\cR ,\rho); \int_R g(r)\rho(dr)=0\}$.\\
  Let us recall some important properties:
  \begin{itemize}
  \item $\forall u\in\bbd, u^{\flat} \in L^2 (\Xi,\cX,\nu ; L_0)\subset L^2 (\Xi\times R ,\cX \otimes \cR
 ,\nu\times\rho).$
 \item $\forall u\in \bbd$, $\int_R \| u^{\flat}\|^2 (\cdot ,r)\rho (dr) =\gamma[u]$.
 \item (chain rule in dim 1) if
$F:\R\rightarrow \R$ is Lipschitz  then $\forall u\in\bbd,\
(F\circ u)^{\flat}=(F'\circ u )u^{\flat}.$ \item (chain rule in
dim $d$) if $F$ is $\mathcal{C}^1$ (continuously differentiable)
and Lipschitz from $\R^d$ into $\R$ then
\[ \forall u=(u_1 ,\cdots ,u_d) \in \bbd^d ,\ (F\circ
u)^{\flat}=\sum_{i=1}^d (F'_i \circ u ) u_i^{\flat}.\]
\end{itemize}
Although not necessary, we assume for simplicity that constants
belong to $\bbd_{loc}$ (see Bouleau-Hirsch \cite{bouleau-hirsch2}
Chap. I Definition 7.1.3.) and that
\begin{equation}\label{232}
1\in \bbd_{loc} \makebox{ which implies }\ \gamma [1]=0 \makebox{
and  } 1^{\flat}=0.
\end{equation}
We now want to iterate the derivative and  define {$\mathbf{ d}^{\infty}$} to
be the set of infinitely differentiable elements in $\bbd$.
We suppose that the bottom structure and $N$ satisfy the following hypothesis:\\
\underline{Hypothesis { (H)}}.
The structure $(\Xi,\cX ,\nu,\bbd,\gamma)$ with generator $(a,\mathcal{D}(a))$ is such that there
exists a subspace $H$ of $\cD (a)\bigcap L^1 (\nu )$,  dense in $L^1 (\nu )\cap L^2 (\nu )$
and such that $\forall f\in H,\;\gamma[f]\in L^2(\nu)$.\\
 {\bf Remark:} In \cite{bouleau-denis, bouleau-denis2}, we assume that the bottom structure satisfies what we call (EID) property. Here, since we  consider 
{\it infinitely differentiable} Poisson functionals, we will be able to prove directly existence and regularity of the density so that we do not need to assume such an hypothesis. 
Moreover, it is obvious that in these references, the (EID) property plays no role in the construction and the properties of the upper structure.
\subsection{Dirichlet structures and Poisson measures.}

 We are given $N$ a Poisson random measure on $[0,+\infty[\times \Xi$ with intensity $dt\times \nu$ defined on the
probability space $(\Om_1 ,\cA_1 ,\bbP_1)$ where $\Om_1$ is the
configuration space, $\cA_1$ the $\sigma$-field generated by $N$ and
$\bbP_1$ the law of $N$. We set $\tN =N-
 dt\times\nu$.\\
 We consider also
another  probability space $(\Om_2,\cA_2 ,\bbP_2)$ on which typically, a semi-martingale or another Poisson measure is defined (see application to SDE's, Section \ref{SDE}).



We shall work on the product probability space:
\[(\Om,\cA ,\bbP)=(\Om_1 \times \Om_2 ,\cA_1 \otimes \cA_2 ,\bbP_1
\times \bbP_2 ).\]

As we have in mind to iterate the derivative, we generalize { the} definition given in \cite{bouleau-denis} in the following way.\\
For all $n\in \Ne$,  we  construct a
random Poisson measure $N\odot\rho^{\odot n}$ on $[0,+\infty[\times \Xi\times R^n
$ with compensator $dt\times\nu\times \underbrace{\rho\times\cdots\times\rho}_{n\ times}$ such that if $N=\sum_i
\varepsilon_{(\alpha_i ,u_i )}$ then $N\odot\rho ^{\odot n}=\sum_i
\varepsilon_{(\alpha_i ,u_i ,R^1_i,\cdots ,R^n_i )}$ where $(R^k_i )_{i\in\Ne}$, $k\in\Ne$,  are independent sequences of
i.i.d. random variables,  independent of $N$ whose common law is
$\rho$ and defined on some infinite product probability space $(\hOm, \hcA ,\hbbP)^{\Ne}$
so that $N\odot\rho^{\odot n}$ is defined on the product probability space:
$(\Om,
\cA ,\bbP)\times (\hOm, \hcA ,\hbbP)^{\Ne}$.\\
In the case $n=1$, we simply denote $N\odot\rho^{\odot 1}=N\odot \rho$.\\
We now introduce the creation and annihilation
operators $\crea$ and $\anni$:
\[\begin{array}{l} \forall (t,u)\in [0,+\infty[\times \Xi,\forall w_1\in\Om_1,\
\crea_{(t ,u)} (w_1)=w_1{\bf 1}_{\{ (t ,u)\in supp
\, w_1\}}+(w_1+\varepsilon_{(t ,u)}\}) {\bf 1}_{\{ (t ,u)\notin supp \, w_1\}}\\
\forall (t,u)\in [0,+\infty[\times \Xi,\forall w_1\in\Om_1,\ \anni_{(t
,u)} (w_1)=w_1{\bf 1}_{\{ (t ,u)\notin supp \,
w_1\}}+(w_1-\varepsilon_{(t ,u)}\}) {\bf 1}_{\{ (t ,u)\in supp \,
w_1\}}.
\end{array}\]
In a natural way, we extend these operators on $\Omega$ by setting
if $w=(w_1 ,w_2)$:
\[ \crea_{(t ,u)}(w)=(\crea_{(t ,u)}(w_1),w_2)\
\makebox{and}\ \anni_{(t ,u)}(w)=(\anni_{(t ,u)}(w_1),w_2),\] and
then to the functionals by
\[ \crea H(w,t ,u)=H(\crea_{(t ,u)} w,t, u)\quad\makebox{ and }\quad\anni H(w,t ,u)=
H(\anni_{(t , u)}w,t,u).\] We now recall the main Theorem of
\cite{bouleau-denis} (see also Theorem 3 in \cite{bouleau-denis2}) which gives a construction and an explicit formula for a
gradient of the upper structure $(\Om ,\cA ,\bbP ,\bbD ,\cE ,\Gamma)$.\\
   We denote by $\cD_0$ the set of elements in $L^2 (\bbP )$
which are the linear combinations of variables of the form
$e^{i\tN (f)}$ with $f\in \left({H}\otimes L^2 (dt)\right)\bigcap L^1 (\nu\times dt)$ {where $H\subset \mathcal{D}(a)$ was introduced in hypothesis (H0)}, recall that $\tN =N-dt\times\nu$,  where ${H}\otimes L^2 (dt)$ denotes the algebraic tensor product of $H$ and $L^2 (dt)$.\\
 Following \cite{bouleau-hirsch2}
  we denote by $L^2 (\R^+ ,dt)\ho \bbd  $ the domain of  the Dirichlet structure which is the product of $(\bbd ,e )$ and the trivial one on $L^2 (dt )$.
  We still denote by $a$ and $\gamma$  its generator and its
  carr\'e du champ, these operators  act only on the  variable $x\in
  \Xi$.\\

If $U=\sum_p \lambda_p e^{i\tN (f_p)}$ belongs to $\cD_0$, we put
\begin{equation}{\label{214}}
\tilde{A}_0 [U]=\sum_p \lambda_p e^{i\tN (f_p)}(i\tN ({a}[f_p
])-\frac12 N (\gamma [f_p])),\end{equation} where, as explained
above, if $f(x,t)=\sum_l u_l (x) \varphi_l (t) \in \cD (a)\otimes
L^2 (dt)$
\[ {a} [f]=\sum_l a[u_l ]\varphi_l \makebox{ and } \gamma [f]=\sum_l \gamma [u_l ]\varphi_l.\]
We denote $\sbbD$ the
completion of $\cD_0\otimes L^2 ([0,+\infty[,dt)\otimes \bbd$ with
respect to the norm
\begin{eqnarray*}\| H\|_{\sbbD}\!\!&=&\!\!\left( \E\int_0^{\infty}\!\!\!\int_{ \Xi }\anni(\gamma
[H])(w,t,u)N(dt,du)\right)^{\frac12} \!+\E\int_0^{\infty}\!\!\!\int_\Xi
(\anni|H|)(w,t,u)\eta (t,u)N(dt,du)\\
\!\!&=&\!\! {\left( \E\int_0^{\infty}\int_{ \Xi }\gamma
[H](w,t,u)\nu (du)dt\right)^{\frac12}\! +\E\int_0^{\infty}\int_\Xi
|H|(w,t,u)\eta (t,u)\nu(du)dt},\end{eqnarray*} where $\eta$ is a fixed
positive function in $L^2 (\R^+ \times \Xi ,dt\times d\nu)$.\\
Finally we denote by $\bbP_N$ the measure $\bbP_N=\bbP (dw)N_w
(dt,du)$. {One has to remember that  the image of $\bbP\times\nu\times dt$ by
$\crea$ is nothing but $\bbP_N$ whose image by $\anni$ is $\bbP\times\nu\times dt$  (see Lemma 13 in \cite{bouleau-denis})}.
\begin{Th}
(i) $A_0$ is symmetric, non positive on $\cD_0$ therefore  it is
closable and we can consider its Friedrichs extension $(A,\cD
(A))$ which generates a closed Hermitian form $\cE$ with domain
$\bbD\supset\cD (A)$ such that
\[\forall U\in\cD (A)\ \forall V\in\bbD ,\ \cE (U,V)=-\E
[A[U]\overline{V}].\] Moreover, $(\cE ,\bbD)$ is a local Dirichlet
form which admits a carré du champ operator
$\Gamma$.\\
(ii) The Dirichlet form $(\bbD ,\cE)$ admits a gradient operator
that we denote by $\sharp$ and given by the following formula:
\begin{equation}\label{formulegradient}
\forall F\in\bbD,\quad\ F^\sharp = \int_0^{+\infty}\int_{\Xi\times R}
\anni((\crea F )^\flat)\, dN\odot \rho\in
L^2(\mathbb{P}\times\hat{\mathbb{P}}) .\end{equation} Formula
{\rm(\ref{formulegradient})} is justified by the following decomposition:
$$F\in\mathbb{D}\quad\stackrel{\crea-I}{\longmapsto}\quad \varepsilon^+F-F\in\underline{\mathbb{D}}
\quad\stackrel{\anni((.)^\flat)}{\longmapsto}\quad\anni((\varepsilon^+F)^\flat)\in
L^2_0(\mathbb{P}_N\times\rho)\quad\stackrel{d(N\odot\rho)}{\longmapsto}\quad F^\sharp\in
L^2(\mathbb{P}\times\hat{\mathbb{P}})$$
where each operator is continuous on the range of the preceding one
and where $L^2_0 (\bbP_N \times \rho )$ is the closed set of
elements $G$ in $L^2
(\bbP_N \times \rho )$ such that $\int_R G d\rho=0$ $\bbP_N$-a.e.\\
Moreover, we have for all $F\in\bbD$
\begin{equation}\label{formuleocc}\Gamma [F]=\hE(F^\sharp )^2
=\int_0^{+\infty}\int_\Xi \anni(\gamma [\crea F ])\,
dN,\end{equation} where $\hE$ denotes the expectation with respect
to probability $\hbbP$.\\
\end{Th}Let us recall  without proof  some
properties of this structure which are quite natural.\\
\begin{Pro}{\label{FormuleDerive}}

 If $h\in L^2 (\R^+,dt) \ho \bbd$, then $\tN
(h)=\int_0^{+\infty}\int_\Xi h(t,u)\tN (ds,du)$ belongs to $\bbD$ and

\item \begin{equation}
\Gamma [\tN (h)]=\int_0^{+\infty}\int_\Xi \gamma
[h(t,\cdot)](u)N(dt,du).\end{equation}
\item \begin{equation}{\label{FormuleUpperSpace}}
\left( \tN (h)\right)^{\sharp}=\int_0^{+\infty}\int_{\Xi\times R}
h^\flat (t,u,r)N\odot\rho (dt,du,dr).
\end{equation}
\end{Pro}
\section{Definition of the Sobolev spaces}
We now give the construction of Sobolev spaces. As this construction is  valid for any Dirichlet form,  we give the details in the case of the bottom structure and then
apply the same procedure to upper space.
\subsection{On the bottom space}
We first define  Hilbert-valued functions in the domain of the Dirichlet form. To this end, let $E$ be a separable Hilbert space.
We denote by $S(E)$ the set of $E$-valued functions defined on $\Xi$ such that there exist $k\in\Ne$, $e_1 $,$\cdots$, $e_k$ in $E$ and $\varphi_1$ ,
$\cdots$, $\varphi_k$ in $\bbd$ with
\[ u=\sum_{i=1}^k \varphi_i e_i.\]
If $u=\sum_{i=1}^k \varphi_i e_i$ belongs to $S(E)$ we can define its derivative $Du=u^\flat=\sum_{i=1}^k \varphi_i^\flat e_i$ as an element of $L^2 (\nu ; L_0\ho E)${, (here and afterwards $\ho$ denotes completed tensor products between Hilbert spaces whereas $\otimes$ denotes algebraic tensor product)}.\\
We denote by $\bbd (E)$ the completion of $S(E)$ with respect to the norm:
\begin{eqnarray}{\label{Defnorm}} \parallel u\|^2_{\bbd (E)}&=&\| u\|^2_{L^2 (\nu)}+\| u^\flat \|^2_{L^2 (\nu ; L_0 \ho E)}.\end{eqnarray}
It is clear that $\bbd (E)$ is a separable Hilbert space and that the linear map $u\in S(E)\mapsto u^\flat \in L^2 (\nu ; L_0 \ho E)$ can be extended in a unique way as a continuous linear map from $\bbd (E)$ into $L^2 (\nu ; L_0 \ho E)$ and that identity \eqref{Defnorm} still holds.\\
From now on, we assume that  $\bbd$ admits a {\it core}. More precisely, we assume:\\
\underline{Hypothesis (C)}: There exists a dense subvector space $\bbd_0 \subset\bbd$ such that each element $u$ in $\bbd_0$ is such that:
\begin{enumerate}
\item $u\in \bigcap_{p\geq 1} L^p (\nu)$.
\item $u$ is {\it infinitely differentiable} in the sense that $u^\flat \in \bbd (L_0 )$, $u^{(2\flat)} =\left( u^{\flat}\right)^\flat \in \bbd(L_0^{\ho 2}),$
    $ \cdots , u^{((n+1)\flat)} =\left( u^{(n\flat)}\right)^\flat \in \bbd(L_0^{\ho (n+1)}),\cdots  $
\item For all $n\in\Ne$, $u^{(n\flat)}\in \bigcap_{p\geq 1} L^p (\nu ; L_0^{\ho n})$.
\end{enumerate}
We introduce: $\bbd_0 (E)=\{ u=\sum_{i=1}^n \varphi_i e_i \in S (E)| \, {\varphi_i \in \bbd_0},\, i=1,\cdots ,n \}.$\\
In a natural way, similarly to the case $n=1$, if $u$ belongs to $\bbd_0 (E)$, for all $n\in\Ne$, $u^{(n\flat )}$ is well-defined as an element in $\bbd ({L_0^{\ho n}\ho E})$.
We are now able to define the Hilbert-valued Sobolev spaces.
\begin{Def} Let $n\in \Ne$, $p\geq 1$. We denote by $\bbd^{n,p}(E)$ the completion of $\bbd_0 (E)$ w.r.t. the norm
\begin{eqnarray}{\label{NormSobolev}} \|u\|_{n,p}&=&\|u\|_{L^p (\nu;E)}+\| u^\flat\|_{L^p (\nu ; L_0\ho E )}+\cdots +\| u^{(n\flat )}\|_{L^p (\nu;L_0^{\ho n}\ho E)}.
\end{eqnarray}
And we set:
\[ \bbd^{\infty} (E)=\bigcap_{n\in\Ne ,p\geq 1} \bbd^{n,p} (E) .\]
\end{Def}
We note $\bbd^\infty$ for $\bbd^\infty(\mathbb{R})$ and $\bbd^{n,p}$ for $\bbd^{n,p}(\mathbb{R})$, the space $\bbd^\infty(E)$ is endowed with the natural inductive limit topology which makes it a Fr\'echet space.

The following facts are standard to prove because $a$ is a closed operator:
\begin{itemize}
\item  Let $n\in\Ne,\, p\geq 1$, $\bbd^{n,p} (E)$ is a Banach space moreover the {operators} $u\mapsto u^\flat$,$\cdots$, $u\mapsto u^{(n\flat )}$ are
well-defined and continuous and equality \eqref{NormSobolev} holds for all $u\in\bbd^{n,p}$.
\item If $n\leq n'$, $p\leq p'$ then {$\bbd^{n',p}(E)\subset \bbd^{n,p}(E)$ and if $\nu (\Xi )<+\infty$, $\bbd^{n',p'}(E)\subset \bbd^{n,p}(E)$}.
\end{itemize}
We have defined the spaces $\bbd^{n,p}(E)$ using the gradient operator. We shall need also more restricted spaces involving the generator :
\begin{Def} We denote by $\bbbd^{\infty}$ the subvector space of elements $u$ in $\bbd^{\infty}$ such that $u$ belongs to 
$\cD (a)$ and $a(u)\in \bbd^{\infty} $ and we consider $$ \bbbd_0 (E)=\{ u=\sum_{i=1}^n \varphi_i e_i \in S(E)| \, \varphi_i \in \bbbd^{\infty},\, i=1,\cdots ,n \}.$$
Moreover, if $u\sum_{i=1}^n \varphi_i e_i \in \bbbd_0 (E)$, we set 
\[ a(u)=\sum_{i=1}^n a(\varphi_i )e_i  \, \in \bbd^{\infty} (E).\]
\end{Def}
\begin{Def}{\label{Def5}} Let $n\in \Ne$, $p\geq 1$. We denote by $\bbbd^{n,p}(E)$ the completion of $\bbbd_0 (E)$ w.r.t. the norm
\begin{eqnarray*}{\label{NormSobolevDom}} \|u\|_{\bbbd^{n,p}}&=&\|u\|_{n,p}+\|a(u)\|_{n,p}\end{eqnarray*}
And we set: $\bbbd^{\infty} (E)=\bigcap_{n\in\Ne ,p\geq 1} \bbbd^{n,p} (E) .$\\
\end{Def}
Here again, the following facts are easy to prove:
\begin{itemize}
\item For all $n\in\Ne,\, p\geq 1$, $\bbbd^{n,p} (E)\subset \bbd^{n,p} (E)\ $.
\item  For all $n\in\Ne,\, p\geq 1$, $\bbbd^{n,p} (E)$ is a Banach space and the map $u\mapsto a(u)$ is well defined and continuous from  $\bbbd^{n,p} (E)$ into $\bbd^{n,p} (E)$.
\item If $n\leq n'$, $p\leq p'$ then {$\bbbd^{n',p}(E)\subset \bbbd^{n,p}(E)$ and if $\nu (\Xi )<+\infty$, $\bbbd^{n',p'}(E)\subset \bbbd^{n,p}(E)$}.
\item $\bbbd^{\infty}=\bbbd^{\infty }({\R}).$
\end{itemize}
As usually, in the case $E=\R$, we omit it in the notations so that by the last property enounced above, there is no ambiguity in the notation.\\
\noindent{\bf Remark:}  We need to introduce $\bbbd^{n,p}$ spaces because in general, an element in $\bbd^{\infty}$ does not belong to $\cD (a)$. Nevertheless, one has to note that in all the classical examples such as the standard Sobolev spaces on $\R^d$ or the Sobolev spaces associated to the {Ornstein-Uhlenbeck} operator this property holds, see Section \ref{Meyer} below. 
\subsection{Sobolev spaces on the upper space}
We define
\[ \bbD_0 =\left\{ \varphi (\tN (f_1),\cdots ,\tN (f_k ))|\; k\in\Ne , \varphi \in C_c^{\infty} (\R^k ), f_i \in L^2 (\R^+ ,dt)\otimes \bbd^{\infty}\, i=1,\cdots ,k\right\},\]
and $$\bbD_0 (E)=\{ \sum_{i=1}^k G_i e_i |\, k\in\Ne , G_i\in \bbD_0 ,\, e_i \in E \ i=1,\cdots ,k\}.$$
As $\bbd_0$ is dense in $\bbd$, $\bbd_{\infty}\supset \bbd_0$ is also dense in $\bbd$ so  $\bbD_0$ is dense in $\bbD$ and clearly each element in $\bbD_0$ belongs to $\bigcap_{p\geq 1} L^p (\bbP)$. Moreover, it is easy to verify that if $X$ belongs to $\bbD_0 (E)$, it is infinitely differentiable. Indeed, similarly to the previous subsection with similar notation, $X^{(n\sharp )}$ is defined inductively as the derivate of $X^{((n-1)\sharp)}\in \bbD \left(L^2 (\hbbP^{(n-1)};E )\right)$ so it belongs to $\bbD \left(L^2 (\hbbP^{ n};E)\right)$ that we consider in a natural way as a subspace of $L^2 (\bbP\times\hbbP^{ n};E)$.
As in the case of the bottom space, we can define for all $n\in\Ne , p\geq 1$ the Sobolev space $\bbD^{n,p} (E)$ which is the closure of $\bbD_0 (E)
$  with respect to the norm
\[ \| X\|_{n,p}=\|X\|_{L^p (\bbP;E)}+\| X^\sharp\|_{L^p (\bbP\times\hbbP ;E)}+\cdots +\| X^{(n\sharp )}\|_{L^p (\bbP\times\hbbP^{n};E)},\]
and $\bbD^{\infty} (E)=\bigcap_{n\in\Ne ,p\geq 1} \bbD^{n,p}(E).$\\
In the same way as in the previous subsection,  for all $n\in\Ne , p\geq 1$  we consider first $\bbbD^{\infty}$, the subvector space of elements in $\bbD^{\infty}\bigcap \cD (A)$ such that $A(X)\in \bbD^{\infty}$, then define in an obvious way $\bbbD_0 (E)$ by 
\[\bbbD_0 (E)=\{ \sum_{i=1}^k G_i e_i |\, k\in\Ne , G_i\in \bbbD^{\infty} ,\, e_i \in E \ i=1,\cdots ,k\}.\]
and finally we construct space $\bbbD^{n,p} (E)$ which is the closure of $\bbbD_0 (E)$ with repect to the norm 
\[ \| X\|_{\bbbD^{n,p} (E)}= \| X\|_{n,p}+ \| A(X)\|_{n,p},\]
and put $\bbbD^{\infty} (E)=\bigcap_{n\in\Ne ,p\geq 1} \bbbD^{n,p} (E).$\\
{\bf Remark:} These Sobolev spaces satisfy the same properties as spaces $\bbd^{n,p}$ and $\bbbd^{n,p}$ {listed after Definition \ref{Def5}} so that we do not recall them.
\begin{Le}{\label{algebre}} Let $X\in \bbD^{\infty} $ and $Y\in\bbD^{\infty}(E)$, then $XY\in\bbD^{\infty} (E)$.
\end{Le}
\begin{proof} Assume first that $X\in \bbD_0$ and $Y\in \bbD_0 (E)$. Then, clearly $Z=XY\in\bbD^{\infty} (E)$ and:
\begin{eqnarray*}
Z ^\sharp (w,w_1 )&=& X^\sharp (w,w_1)Y(w)+X(w)Y^\sharp (w,w_1 ),\\
Z^{(2\sharp )} (w,w_1 ,w_2 )&=& X^{(2\sharp )}(w,w_1,w_2)Y(w)+X^\sharp (w,w_1)Y^\sharp (w,w_2)+\\&&X^\sharp (w,w_2)Y^\sharp (w,w_1 )+X(w)Y^{(2\sharp )}(w,w_1,w_2 )
\end{eqnarray*}
and more generally, for any $n\in \Ne$, $X^{(n\sharp)} $ can be expressed as the sum of $2^n$ terms of the form $X^{(k\sharp)}Y^{((n-k)\sharp)}$ with $k\in
\{ 0,\cdots ,n \} $. The Holder's inequality yields for all $p\geq 1$:
\[ \| Z^{(n\sharp )}\|_{L^p (\bbP\times \hbbP^n ;E)} \leq 2^n \| X\|_{n,2p}\| Y\|_{n,2p}.\]
We deduce that for all $n\in\Ne$ and all $p\geq 1$, there exists a constant $C_{n,p}$ such that,
\[ \|Z\|_{n,p}\leq C_{n,p}\| X\|_{n,2p}\| Y\|_{n,2p}.\]
We then conclude by density.
\end{proof}
\begin{Pro}{\label{Module}} Let $X\in\bbD^{\infty} (E)$ then $Y=\| X\|_E^2$ belongs to $\bbD^{\infty}$.
\end{Pro}
\begin{proof} Assume first that $X\in \bbD_0 (E)$: $ X= \sum_{i=1}^k G_i e_i ,$ 
where, without loss of generality, $(e_i )_{ 1\leq i\leq k}$ is an orthonormal family in $E$ so that 
 $ Y =\sum_{i=1}^k G_i^2 .$\\
As a consequence of the previous Lemma, $Y$ belongs to $\bbD^{\infty}$ and by the functional calculus, we have:
\[ Y^\sharp =2\sum_{i=1}^k G_i G_i^\sharp .\]
Let $p\geq 1$, then obviously: $ \| Y\|_{L^p (\bbP )} =\| X\|_{L^{2p}(\bbP ;E)}^2$ 
and using the trivial inequality $2\left|\sum_{i=1}^k G_i G_i^\sharp\right| \leq \sum_{i=1}^k G_i^2 +\sum_{i=1}^k (G_i^\sharp )^2 $ we get
\begin{eqnarray*}
\| Y^\sharp \|_{L^p (\bbP\times \hbbP )} &\leq&\| X\|_{L^{2p}(\bbP;E)}^2 +\| X^\sharp \|^2_{L^{2p} (\bbP\times\hbbP;E)}.
\end{eqnarray*}
As a consequence: 
$
\| Y\|_{1,p}\leq 2\| X\|_{1,2p}^2.
$\\
More generally, let $n\in \Ne$, we have
\begin{eqnarray*} 
(G_i^2 )^{(n\sharp)}(w,w_1 ,\cdots ,w_n )&=&\sum_{j=1}^{2^n}G_i^{(m_j \sharp )} (w, w_{\sigma_j (1)},\cdots ,w_{\sigma_j (m_j )})G_i^{((n-m_j)\sharp )} (w, w_{\sigma_j (m_j +1 )},\cdots ,w_{\sigma_j (n)}),\end{eqnarray*}
where for all $j\in \{ 1,\cdots ,2^n \}$, $m_j \in \{ 0,\cdots ,n\}$, $\sigma_j$ is a permutation on $\{ 1,\cdots ,n\}$ and both do not depend on $G_i$. 
Using  $ab\leq \frac12 (a^2 +b^2)$ we get:
\begin{eqnarray*} \sum_{i=1}^k (G_i^2 )^{(n\sharp)}&\leq& \frac12 \sum_{j=1}^{2^n}\sum_{i=1}^k \left( (G_i^{(m_j \sharp )})^2 +(G_i^{((n-m_j)\sharp )})^2 \right) \\
&=& \frac12\sum_{j=1}^{2^n}\left(  \| G^{((m_j)\sharp)}\|_E^2 +\| G^{((n-m_j)\sharp)}\|_E^2 \right) .\end{eqnarray*} 
This yields:
\[ \| Y^{n\sharp}\|_{L^p (\bbP\times\hbbP^n )}\leq 2^n \| G\|_{n,2p}^2.\]
We deduce that for all $n\in \Ne$ and all $p\geq 1$ there exists a constant $C_{n,p}$ such that 
\[ \| Y\|_{n,p} \leq C_{n,p}\| X\|^2_{n,2p}.\]
It is now easy to conclude using a density argument.
\end{proof}
\begin{Cor}\label{Gamma}  Let $X\in {\bbD^\infty}$, then $\Gamma [X]$ belongs to ${\bbD^\infty}$.
\end{Cor}
\begin{proof} Just apply the preceding Proposition to $X^\sharp$.
\end{proof}
The next Lemma generalizes identity \eqref{FormuleUpperSpace}:
\begin{Le}{\label{DeriveSimple}}  Let $h\in L^2 (\R^+,dt)\otimes \bbd^{\infty}$, then $\tN
(h)=\int_0^{+\infty}\int_\Xi h(t,u)\tN (ds,du)$ belongs to $\bbD^{\infty}$ and for all $n\in\Ne$ :
\begin{equation}
\tN (h)^{(n\sharp )}=\int_0^{+\infty}\int_{\Xi\times R^n} 
h^{(n\flat )} (t,u,r_1,\cdots ,r_n )N\odot\rho^{ \odot n}(dt,du,dr_1 ,\cdots ,r_n).
\end{equation}
\end{Le}
\begin{proof} In fact, this is a direct consequence of Proposition \ref{FormuleDerive}. Indeed, for all $n\in\Ne$, take for {\it bottom space } 
the product space $\Xi\times R^n$ equipped with the Dirichlet structure which is the product structure of $(\bbd ,e)$ with the trivial ones on $L^2 (\rho^n )$. Then, 
following the same construction as above with $N$ replaced by $N\odot \rho^{\odot n}$, we obtain a Dirichlet structure on $L^2 (\bbP\times\hbbP^n )$. It is 
obvious that this structure is the product of $(\bbD ,\cE )$ with the trivial one on $L^2 (\hbbP ^n )$ so  we identify it with $\bbD (L^2 (\hbbP ^n ))$. Consider now $h\in L^2 (\R^+,dt)\otimes \bbd^{\infty}\bigcap L^1 (dt\times \nu \times \rho^n )$ then, 
as $h^{(n\flat )}$ takes its value in $L_0^n $, $N\odot \rho^{\odot n} (h^{(n\flat )})=\widetilde{N\odot \rho^{ \odot n}} (h^{(n\flat )})$, it is now easy to apply Proposition \ref{FormuleDerive} to get the expression of $\tN (h)^{((n+1)\sharp)}$ in this case and then to conclude using a density argument.
\end{proof}
\section{Identity of $\mathbb{D}^\infty(E)$ and $\overline{\mathbb{D}^\infty}(E)$ and Meyer inequalities in the classical cases}{\label{Meyer}}
This section is devoted to inequalities in $L^p$ norms. The first subsection gives a general  equivalence between gradient and carr\'e du champ operator due to Khintchine inequality and some of its improv{e}ments. The following subsections deal with the classical cases where the bottom space is either the Euclidean space equipped with the Laplacian or the Wiener space equipped with the Ornstein-Uhlenbeck operator.

\subsection{An equivalence of norms} 
Let us first emphasize that when we write $F^{\sharp\sharp}$ the second $\sharp$-operator acts on $F^\sharp(\omega, \hat{\omega}_1)$ with fixed $\hat{\omega}_1$ and adds a new $\hat{\omega}_2$ independently. We write $\hat{\mathbb{E}}$ for the expectation w.r. to all these $\hat{\omega}_1$, $\hat{\omega}_2$ etc, in other words $\hE$ denotes the expectation with respect to $\hbbP^{\otimes\Ne}$.\\
Now, we introduce the following notation for any $F\in \bbD^{k,2}$:
\begin{equation}\label{gammak}
\Gamma_k[F]=\hat{\mathbb{E}}[(F^{(k\sharp )})^2].
\end{equation}  This is a general definition of carr\'e du champ operators of order $k$ (cf. P.-A. Meyer s\'em XVIII \cite{meyer18} p182 in the Ornstein-Uhlenbeck case where operators $\Gamma_k$ satisfy a specific recurrence relation due to a commutation identity, that we do not suppose in this subsection).\\
The aim of this subsection is to prove that by choosing well the probability space $(R,\cR, \rho)$ and the version of the gradient on the bottom space then 
the norm on $\bbD^{k,p}$ is equivalent to the following  norm
$$ \| F\|_{L^p (\bbP )}+\sum_{i=1}^k \| (\Gamma_i [F])^{1/2}\|_{L^p (\bbP )}.$$

\begin{Pro} \label{equiv-p} We can choose $(R,\cR ,\rho)$ and the gradient operator $\flat$ such that for all $k\in\Ne$, $p> 1$ the following inequality holds for any $F\in\bbD^{k,p}$: 
\begin{equation}\label{occ-Lp}
c_{p,k}(\Gamma_k[F])^{1/2}\leq \|F^{(k\sharp )}\|_{L^p(\hat{\mathbb{P}}^k) }\leq C_{p,k}(\Gamma_k[F])^{1/2}\ {\quad \mathbb{P}\mbox{-a.s.}},
\end{equation}
where $c_{p,k}$ and $C_{p,k}$ are  constants only depending on $p$ and $k$.
\end{Pro}
\begin{proof} Take $(R,\cR ,\rho)$ such that a sequence $(\xi_i )_{i\in\Ne}$ of Rademacher functions may be defined on it. The simplest choice is to take 
$R=[0,1]$, $\rho$ the Lebesgue measure and for $(\xi_i )$ the standard Rademacher functions. We recall that $(\xi_i )_i$ is a sequence of i.i.d. variables 
defined on $L^2 (\rho )$ and such that $\rho (\xi_i =1)=\rho (\xi_i =-1 )=1/2$. Then we choose the version of the gradient $\flat$ with values in $V$, the vector space spanned by the  $(\xi_i )_i$. It is clear that $V\subset L^2_0 (\rho )$. Hence we have the following decomposition for any $f\in\bbd^{\infty}$
\begin{eqnarray}{\label{gradBase1}}f^{(\flat )} (u,r)=\sum_{i} \langle f^{(\flat )} (u,\cdot ),\xi_{i}\rangle_{L^2 (\rho )}\xi_{i} (r).\end{eqnarray}
Consider now $F=\varphi (\tN (f_1 ),\cdots,\tN (f_n))$ in $\bbD_0$, then
$$F^{(\sharp )} =N\odot \rho (J_1 ),$$
where $
J_1 (t,x,r)=\sum_{j=1}^n \partial_j \varphi(\tN (f_1 ),\cdots,\tN (f_n)) f_i (t,x,r).$\\
Applying \eqref{gradBase1} , this yields
$$J_1 (t,u,r)=\sum_{i} D_{i} (t,u)\xi_{i}(r),$$
where $D_{i}(t,u)=\langle J_1 (t,u,\cdot ),\xi_{i}\rangle_{L^2 (\rho )}$.\\
 We have
\begin{eqnarray*}
F^\sharp &=&  N\odot \rho (J_1 )
=\sum_{j=1}^{Y}J_1(T_j ,X_j, R_j )=\sum_{j=1}^{Y} \sum_{i=1}^{+\infty}D_i (T_j ,X_j )\xi_i (R_j )
\end{eqnarray*}
where  $(T_j ,X_j )$ is the sequence of marked points of the random measure $N$ and $Y$ is the number of jumps in $\Xi$ before $T$ ($Y=+\infty$ a.s.  iff $\nu (\Xi )=+\infty$) and $(R_j)_{ j\geq 1}$ is a  sequence of independent random variables with common law $\rho$ independent of $(T_j ,X_j )$ (see the construction of $N\odot\rho$ above). But, it is obvious that the sequence 
$\left( \xi_{i}(R_j)\right)_{i,j\in\Ne }$ is a Rademacher sequence of independent variables defined on the probability space $L^2 (\hat{\Omega},\hbbP )$ so independent of the $(T_j, X_j )_{j\in\Ne}$ and of $Y$. Then Khintchine inequality (cf. \cite{ledoux-talagrand} p. 91) yields: 
$$ A_p \| F^\sharp\|_{L^2 (\hbbP )}\leq \| F^\sharp\|_{L^p (\hbbP )}\leq B_p \| F^\sharp\|_{L^2 (\hbbP)},$$
in other words:
$$A_p \left( \Gamma [F]\right)^{1/2}\leq \| F^{\sharp }\|_{L^p (\hbbP)}\leq B_p \left( \Gamma [F]\right)^{1/2}\quad  {\quad \mathbb{P}\mbox{-a.s.}}$$
where $A_p=c_{p,1}$ and $B_p=C_{p,1}$ are the constants in Khintchine's inequality. Using a density argument, we get the result in the case $k=1$.\\
For $k>1$ we {cannot follow strictly the same argument because Khintchine inequalities are not relevant in that case (it seems that some authors do not detect this difficulty, which demands using stronger inequalities, see below).}
As above we first consider $F=\varphi (\tN (f_1 ),\cdots,\tN (f_n))$ in $\bbD_0$.  By iteration:
$$F^{(k\sharp )} =N\odot \rho^{\odot k} (J_k ),$$
where $J_k \in L^2 (\Omega )\otimes L^2 ([0,T])\otimes \bbd^{\infty} (L^2 (R^k ))$. \\
We have:
$$J_k (t,u,r_1 ,\cdots ,r_k )=\sum_{i_1 ,\cdots ,i_k} D_{i_1 ,\cdots ,i_k} (t,u)\xi_{i_1}(r_1)\xi_{i_2}(r_2 )\cdots \xi_{i_k}(r_k),$$
where $D_{i_1 ,\cdots ,i_k}(t,u)=\langle J_k (t,u,\cdot ),\xi_{i_1}\xi_{i_2}\cdots \xi_{i_k}\rangle_{L^2 (\rho^k )}$.\\
Then, with the same notations as above:
\begin{eqnarray*}
F^{(k\sharp)} &=&  N\odot \rho^{\odot k} (J_k )=\sum_{j=1}^{Y}J_k(T_j ,X_j, R^{1}_j ,R^{2}_j ,\cdots ,R^{k}_j )\\ &=&\sum_{j=1}^{Y} \sum_{i_1 ,\cdots ,i_k}D_{i_1 ,\cdots ,i_k} (T_j ,X_j )\xi_{i_1} (R^{1}_j )\xi_{i_2} (R^{2}_j )\cdots \xi_{i_k} (R^{k}_j).
\end{eqnarray*}
Here $(R_j^{l})_{ j\geq 1}$, $l\in \{ 1,\cdots ,k\}$ are $k$ independent  sequences of independent random variables with common law $\rho$ thus the sequence 
$\left( \xi_{i}(R^{l}_j)\right)_{l\leq k; i,j\geq 1 }$ is a Rademacher sequence of independent variables defined on the probability space $L^2 (\hat{\Omega}^k ,\hbbP^k )$.
To conclude as in the case $k=1$, we need some Khintchine's inequalities for $k$-fold products of Rademacher functions. These inequalities are due to C. Borell and A. Bonami. More precisely, as a consequence of Theorem 1.1. in \cite{pisier}, we have:\\ 
If $p\in ]1,2[$:\\
$$(p-1)^{\frac{k}{2}}\| F^{(k\sharp )}\|_{L^2 (\hbbP^k)}\leq \| F^{(k\sharp )}\|_{L^p (\hbbP^k)}\leq  \| F^{(k\sharp )}\|_{L^2 (\hbbP^k)}$$
i.e.
$$ (p-1)^{\frac{k}{2}}\left( \Gamma_k [F]\right)^{1/2}\leq \| F^{(k\sharp )}\|_{L^p (\hbbP^k)}\leq \left( \Gamma_k [F]\right)^{1/2}\quad {\quad \mathbb{P}\mbox{-a.s.}}$$
And if $p>2$:
$$\| F^{(k\sharp )}\|_{L^2 (\hbbP^k)}\leq \| F^{(k\sharp )}\|_{L^p (\hbbP^k)}\leq (p-1)^{\frac{k}{2}} \| F^{(k\sharp )}\|_{L^p (\hbbP^k)}$$
i.e.
$$\left( \Gamma_k [F]\right)^{1/2}\leq \| F^{(k\sharp )}\|_{L^p (\hbbP^k)}\leq  (p-1)^{\frac{k}{2}}\left( \Gamma_k [F]\right)^{1/2}\quad {\quad \mathbb{P}\mbox{-a.s.}}$$
So in all the cases, we have proved the desired inequality for $F\in\bbD_0$,  it is easy to conclude by density.\end{proof}
\noindent {\bf Remark:} 
There is a strong similarity about the Khintchine-like inequalities between Rademacher functions and Gaussian variables. Now, in the case $k=1$ the argument above extends to 
a Gaussian basis and gives even a stronger result:
$$ \| F^{\sharp }\|_{L^p (\hbbP)}=\kappa_p \left( \Gamma [F]\right)^{1/2},$$
where $\kappa_p =\left( \frac{1}{\sqrt{2\pi}}\int_{\R}|x|^p e^{-x^2 /2}\, dx\right)^{1/p}$.\\
 In the case $k>1$, inequalities like Borell-Bonami for Gaussian variables are known (see for example \cite{latala}) hence yield the same result.
\subsection{ Meyer inequalities in the Euclidean case}{\label{4.2}}
We consider now the classical case where $\nu$ is the Lebesgue measure $dx$ on $\mathbb{R}^d$ and $a=\frac{1}{2}\Delta$.

We  start from the Stein inequalities  about Riesz transfoms \cite{stein}. For $f$ say in $\mathcal{S}(\mathbb{R}^d)$ the Riesz transforms are defined by the Fourier multiplier 
$$\widehat{R_jf}(\xi)=\frac{i\xi_j}{(\sum_1^d\xi_j^2)^{1/2}}\hat{f}(\xi)$$
i.e. symbolically  $$
R_jf=\frac{\partial}{\partial x_j}(\frac{1}{\sqrt{-\Delta}}f)$$
\begin{Th}{\rm(E.M. Stein)} For $1<p<\infty$, there are constants $C_p$, $c_p$, depending only on $p$ such that $\forall f\in \mathcal{S}(\mathbb{R}^d)$ 
\begin{equation}\label{stein}
c_p\|f\|_p\leq\|(\sum_1^d|R_jf|^2)^{1/2}\|_p\leq C_p\|f\|_p
\end{equation}
\end{Th}
\noindent  Applying (\ref{stein}) to $\sqrt{-\Delta}f$ we have
\begin{equation}\label{stein2}
c_p\|\sqrt{-\Delta}f\|_p\leq\|(\sum_1^d|\partial_jf|^2)^{1/2}\|_p\leq C_p\|\sqrt{-\Delta}f\|_p
\end{equation}
Let us remark that the functions $F=e^{iNf}$ with $f$ infinitely differentiable with support in a set $A$ of finite Lebesgue measure generate an algebra $\mathbf{ L}$ dense in every $L^p(\mathbb{P})$, cf Lemma 8 of \cite{bouleau-denis}.

Such a random variable may be written 
$F=\sum_k1_{\{Y=k\}}F_k(X_1,\ldots,X_k)$ where $Y$ follows a Poisson distribution with parameter $|A|$ and where the $X_\ell$ are i.i.d. with distribution $\frac{1}{|A|}dx|_A$.

In the present classical case the operators $\nabla$ and $\Delta$ tensorize so that we have $\Gamma[F]=\sum_k1_{\{Y=k\}}\sum_{j=1}^k(\partial_jF_k)^2(X_1,\ldots,X_k)$ and
 $AF=\sum_k1_{\{Y=k\}}\Delta F_k(X_1,\ldots,X_k)$. 
 
 This implies $\sqrt{-A}F=\sum_k1_{\{Y=k\}}\sqrt{-\Delta} F_k(X_1,\ldots,X_k)$. So that we can write
$$
\mathbb{E}|\Gamma[F]|^p=\sum_k\mathbb{P}{\{Y=k\}}\mathbb{E}|\sum_{j=1}^k(\partial_jF_k)^2(X_1,\ldots,X_k)|^{p/2}
\leq C_p^p\sum_k\mathbb{P}{\{Y=k\}}\mathbb{E}|\sqrt{-\Delta_k}F_k|^p$$
where we have applied the inequality  (\ref{stein2}) in dimension $kd$ with the measure $(\frac{1}{|A|}dx|_A)^k$ what does not change the constants. So
$$\mathbb{E}|\Gamma[F]|^p\leq 
C_p^p\mathbb{E}|\sum_k1_{\{Y=k\}}\sqrt{-\Delta_k}F_k|^p=C^p_p\mathbb{E}|\sqrt{-A}F|^p
$$
and finally by density and similarly on the left hand side, we obtain 
\begin{Pro}\label{pro19} a) $ \forall F\in\mathbb{D}$, with the same constants as in {\rm(\ref{stein})}
\begin{equation}\label{meyer-poisson}
c_p\|\sqrt{-A}F\|_p\leq\|\sqrt{\Gamma[F]}\|_p\leq C_p\|\sqrt{-A}F\|_p
\end{equation}
b) For $F$ in the closure of $\mathbf{L}$ for the norm $\|F^{k\sharp}\|_p$, 
\begin{equation}\label{inegalitesdemeyer}c_p\|(-A)^{k/2}F\|_p\leq\|\sqrt{\Gamma_k[F]}\|_p\leq C_p\|(-A)^{k/2}F\|_p.\end{equation}
\end{Pro}
The proof of part b) of the proposition follows exactly the same lines thanks to the representation $F=\sum_k1_{\{Y=k\}}F_k(X_1,\ldots,X_k)$ on the algebra $\mathbf{ L}$.\\
\noindent {\bf Remarks:} \\
1- This inequality has been obtained by  L. Wu \cite{wu1} following a way suggested by P.-A. Meyer of approximating $\mathbb{R}^d$ by the torus.\\
2- Let us emphasize that the tensorization property of $\nabla$ and $\Delta$ is crucial in this result. For example if we had started with the operator $\sqrt{-\Delta}$ on the bottom space in the role of $a$, we would not have obtained on the upper space the operator $\sqrt{-A}$. Subordination doesn't commute with tensorization.
\\
3- Inequalities (\ref{inegalitesdemeyer}) and (\ref{occ-Lp}) imply the identity of the spaces $\mathbb{D}^\infty(E)$ and $\overline{\mathbb{D}^\infty}(E)$.
\subsection{Other cases of transfer of inequalities on the Poisson space}{\label{4.3}}
a) Suppose that the bottom space is an abstract Wiener space equipped with the Ornstein-Uhlenbeck structure, then Meyer inequalities hold on all tensor products of the bottom space which are still abstract Wiener spaces with the same constants. In that case where the bottom space is a probability space, Meyer's inequalities are preserved by product and lift to the Poisson space.

 In this case, the logarithmic Sobolev inequalities hold on the bottom space. These inequalities tensorize with the same constants (cf. for instance \cite{ane-et-al}). The measure $\nu$ is finite and the semigroup $P_t$ associated with $A$ acts  as $P_t F=\sum_k1_{\{Y=k\}}P_t[F_k(X_1,\ldots,X_k)]$. The logarithmic Sobolev inequalities hold above. This shows that the hypercontractivity property holds on the Poisson space.

\noindent b) Suppose the bottom space satisfies the Bakry hypothesis
$$a\gamma[f]-2\gamma[f,af]\geq 0$$
then (cf Bakry \cite{bakry} and Bakry-Emery \cite{bakry-emery}) the Meyer inequalities hold down. Then, if 

(i) the operator $a$ tensorizes, i.e. on  functions $\sigma(N|_B)$-measurable for some set $B$ such that $\nu(B)<\infty$, the operator $A$  writes $AF=\sum_k1_{\{Y=k\}}(a_kF_k)(X_1,\ldots,X_k)$ where $$a_k F_k=a[F(.,X_2,\ldots,X_k)](X_1)+\cdots+a[F_k(X_1,\ldots,X_{k-1},.)](X_k)$$
\indent(ii) the inequalities hold on the products on the bottom space with the same constants, 

\noindent then the Meyer inequalities transfer on the Poisson space.\\
\noindent{\bf Remark:}  In the cases considered in Subsection \ref{4.2} and in this subsection, spaces $\mathbb{D}^\infty(E)$ and $\overline{\mathbb{D}^\infty}(E)$ are identical. The identity of these spaces  extends of course  to the situation where the bottom space is  a product with a new  factor carrying a null Dirichlet form. We will have such a case in the examples below.

\section{Criterion of smoothness for the law of Poisson functionals}
\begin{Le}{\label{Inverse}} Let $X\in \bbD^{\infty}$ be positive and such that $\displaystyle\frac1X \in \bigcap_{p\geq 1} L^p (\bbP)$, then
 \[\displaystyle\frac1X \in\bbD^{\infty}.\]
 \end{Le}
 \begin{proof} For all $\varepsilon >0$, we put:
 \[ X_{\varepsilon}=\displaystyle\frac{1}{\varepsilon +X}.\]
 As the map $\varphi_{\varepsilon}:x\mapsto \frac{1}{\varepsilon +x}$ is infinitely differentiable on $[0 ,+\infty[$ with 
 bounded derivatives of any order, $X_{\varepsilon}$ belongs to $\bbD^{\infty}$. Indeed, first $\varphi_{\varepsilon} (X)$ belongs to 
 $\bbD$ and by the chain rule and the hypotheses we made:
 \[ X_{\varepsilon}^\sharp =\displaystyle\frac{-X^{\sharp}}{(\varepsilon+X)^2}\in\bigcap L^p (\bbP \times\hbbP ),\]
 so $X_{\varepsilon}\in \bigcap_{p\geq 1} \bbD^{1,p}$.\\
 As $X^\sharp \in \bbD (L^2 (\hbbP))$ we deduce as a consequence of Lemma \ref{algebre} that $X_{\varepsilon}^\sharp$ belongs to 
 $\bbD (L^2 (\hbbP ))$ and still by the chain rule:
 \[ X_{\varepsilon}^{2\sharp}=\displaystyle\frac{-X^{(2\sharp )}}{(\varepsilon+X)^2}+2\displaystyle\frac{X^{\sharp}X^\sharp}{(\varepsilon+X)^3}\in \bigcap_{p\geq 1} L^p (\bbP\times \hbbP^{ 2}),\]
 this ensures that $X_{\varepsilon} \in \bigcap_{p\geq 1} \bbD^{2,p}$.\\
 Then by iteration, we obviously get that  $X_{\varepsilon}^{(n\sharp )}\in \bigcap_{p\geq 1} L^p (\bbP \times \hbbP^{ n})$ for all $n\in\N$ so $X_{\varepsilon}\in\bbD^{\infty}$.\\
  Moreover, by the dominated convergence theorem, it is clear that for all $n\in\N$, $X_{\varepsilon}^{(n\sharp )}$ converges in $L^p (\bbP \times \hbbP^{n})$ as $\varepsilon$ goes to $0$. So $X_{\varepsilon}$ converges in $\bbD^{\infty}$ equipped with its natural Fr\'echet topology to an element which is nothing but $X$ and this ends the proof.
\end{proof}
This yields the following corollary:
 \begin{Cor}{\label{inverseGamma}}
Let $d\in\Ne$ and $X\in (\bbD^{\infty})^d$, if $det (\Gamma [X])>0$ and if $\displaystyle\frac{1}{det (\Gamma [X])}$ belongs to $\bigcap_{p\geq 1} L^p (\bbP)$ then
\[\left( \Gamma [X]\right)^{-1}\in (\bbD^{\infty})^{d\times d} \makebox{ and } \left( \Gamma [X]\right)^{-1}\cdot{X^\sharp}\in \left(\bbD \left(L^2 (\hbbP )\right)\right)^d.\]
\end{Cor}
\begin{proof} First of all, it is clear that $det(\Gamma [X])$  belongs to $\bbD^{\infty}$ thanks to Lemma
\ref{algebre}. Applying the previous Lemma, we conclude that $\displaystyle\frac{1}{det(\Gamma [X])}\in \bbD^{\infty}$. From this, as $(\Gamma[X])^{-1}$ is the product of $\displaystyle\frac{1}{det(\Gamma [X])}$ and the co-factors matrix, it is clear that
$(\Gamma[X])^{-1}$ belongs to $(\bbD^{\infty})^{d\times d}$. \\
The second property is a direct consequence of Lemma \ref{algebre}.
\end{proof}
In view of the next Lemma, we recall that  $A$ denotes the generator of the Dirichlet form $(\bbD ,\cE )$. \\
The operator $X\mapsto X^\sharp$, considered as an unbounded operator with domain $\bbD \subset L^2 (\bbP )$ and values in   $L^2 (\bbP\times \hbbP)$, admits an adjoint operator that we denote by $\delta$. Its is an operator with domain $\cD (\delta ) \subset L^2 (\bbP\times \hbbP)$ and values in $L^2 (\bbP )$.
\begin{Le}{\label{FormuleAdj}} Let $X\in \bbD^{\infty}$ and $Y\in \bbbD^{\infty}$  then $XY^\sharp$ belongs to $\cD (\delta )$ and
\[ \delta [XY^\sharp ]=-2XAY-\Gamma [X,Y].\]
\end{Le}
\begin{proof} Let $Z\in\bbD^{\infty}$ then $XZ \in\bbD^{\infty}$ and by definition of $A$, we have:
\begin{eqnarray*}
\E\hE [ (ZX)^\sharp Y^\sharp ]&=&\E [\Gamma [ZX,Y]]\\
&=& \E [ZX (-2 AY)].
\end{eqnarray*}
But $(ZX)^\sharp =Z^\sharp X + ZX^{\sharp}$ so that
\begin{eqnarray*}
\E\hE [ Z^{\sharp} XY^{\sharp}]&= &\E [ZX (-2 AY)]-\E\hE [Z X^{\sharp}Y^\sharp ]\\
&=& \E [ZX (-2 AY)]-\E[Z \Gamma [X,Y]].
\end{eqnarray*}
We end the proof using the fact that $\bbD^{\infty}$ is dense in $L^2 (\bbP )$.
\end{proof}
We now turn out to the main result of this section which gives a
criterion of smoothness for an element in $\bbD^{\infty}$, we
shall apply it to several examples in the next sections.
\begin{Pro}\label{Densite} Let $d\in\Ne$ and $X$ be  in $\left(\bbbD^{\infty}\right)^d$. If
$\left( \Gamma [X]\right)^{-1} \in \bigcap_{p\geq 1} L^p
 (\bbP; \R^{d\times d} )$, 
 then $X$ admits a density which belongs to $C^{\infty}(\R^d )$.
 \end{Pro}
 \begin{proof} Let $f\in C^{\infty}_c (\R^d )$. With obvious
 notation, we consider the column vector:
 \[ \Gamma [f(X),X]=\left( \Gamma [f(X),X_i ]\right)_{1\leq i\leq
 d}
 .\]
 As a consequence of the functional calculus related to the local
 Dirichlet forms (see \cite{bouleau-hirsch2}, section I.6) we
 have for all $i\in \{ 1,\cdots ,d\}$:
 \begin{eqnarray*}
\Gamma [f(X),X_i]&=& \sum_{j=1}^d {\partial_j}f(X)\Gamma [X_j ,X_i
]
\end{eqnarray*}
so that 
$
\Gamma [f(X),X]= \Gamma [X]\nabla f(X)
$ 
  and $ \nabla f(X)=\left(\Gamma [X]\right)^{-1}\Gamma [f(X),X].$\\
   We now denote by $e_i$ the $i-$th column vector of the canonical
 basis in $\R^d$ and by $M^\ast$ the transposed  of any matrix
 $M$. We have:
 \begin{eqnarray*}
 \bbE[\partial_i f(X)]&=& \E[\nabla f (X)^\ast e_i ]
 =\bbE[\Gamma [f(X),X^\ast ]\left(\Gamma [X]\right)^{-1}e_i ]\\
 &=&\bbE\hE [ f(X)^\sharp X^{\ast,\sharp}\left( \Gamma [X]\right)^{-1}e_i
 ]= \bbE \left[f(X) \delta [X^{\ast,\sharp }\left(\Gamma [X]\right)^{-1}e_i ]
 \right]\end{eqnarray*}
 If $i_1 ,i_2 ,\cdots ,i_n ,\cdots $ is a given sequence  in $\{ 1,\cdots ,d\}$ we get
 by iteration:
 \begin{eqnarray*}
 \bbE [\partial_{i_n}\cdots \partial_{i_1} f(X)]&=&\\
 &&\hspace{-2cm}\bbE \left\{
 f(X)\delta\left[ X^{\ast,\sharp }\left(\Gamma
 [X]\right)^{-1}e_{i_n}\delta \left[ X^{\ast,\sharp }\left(\Gamma
 [X]\right)^{-1}e_{i_{n-1}}\delta\left[ \cdots \delta \left[ X^{\ast,\sharp }\left(\Gamma
 [X]\right)^{-1}e_{i_1}\right] \right.\cdots \right.
 \right]\right\}\end{eqnarray*}
 More precisely, we have for all $n\in\Ne$:
 \begin{eqnarray}{\label{eqZ}}\bbE \left[\partial_{i_n}\cdots \partial_{i_1} f(X)\right]&&=\bbE
 \left[ f(X)Z_n \right],\end{eqnarray}
 where $Z_n $ is defined inductively by :
 \[ \left\{ \begin{array}{rcl}
 Z_1&=& \delta [ X^{\ast,\sharp }\left(\Gamma
 [X]\right)^{-1}e_{i_1}]\\
 Z_n &=& \delta [ X^{\ast,\sharp }\left(\Gamma
 [X]\right)^{-1}e_{i_{n}}Z_{n-1}],\ \ n\in\Ne.\end{array}\right. \]
By Lemma \eqref{FormuleAdj}, we obtain that for all $n\in\Ne$:
\[ Z_n =-2A[X^\ast ]\left(\Gamma
 [X]\right)^{-1}e_{i_{n}}Z_{n-1}-\sum_{j=1}^d \Gamma [X^{\ast}_j, a_{j,i_n }Z_{n-1}],\]
 where $a_{j, i_n}$ denotes the $j-$th element of the $i_n$ column of the matrix $(\Gamma [X])^{-1}$.
As $A[X]\in (\bbD^{\infty} )^d$ and thanks to Lemmas \ref{algebre}, \ref{Gamma} and \ref{Inverse} we conclude that for all $n\in\Ne$,
$Z_n$ belongs to $\bbD^{\infty}$ hence in $L^1 (\bbP )$. \\
So, equality \eqref{eqZ} implies that for all $n\in\Ne$ and all $f\in C^{\infty}_c (\R^d )$:
\[ \E [|f^{(n)} (X)|]\leq \| f\|_{\infty} \E [|Z_n|].\]
This ends the proof by standard arguments.
\end{proof}

\section{Application to Poisson driven sde's}{\label{SDE}}
\subsection{The SDE we consider}
We assume that on the probability space $(\Om_2,\cA_2 ,\bbP_2)$, an
$\R^n$-valued semimartingale $Z=(Z^1 ,\cdots,Z^n)$ is defined,
$n\in\Ne$.
As in \cite{bouleau-denis2}, we adopt the following assumption  on the bracket of $Z$ and on the total variation of its finite variation part. It is satisfied if both are dominated by the Lebesgue measure uniformly:\\
 {\underline{Assumption on $Z$}}:

There exists a positive constant $C$ such that  for any square integrable $\R^n$-valued predictable process $h$:
\begin{equation}\label{CondCrochet} \forall t\geq 0 ,\ \E  [(\int_0^t h_s dZ_s )^2]\leq C^2 \E [\int_0^t |h_s|^2 ds].\end{equation}
Let $d\in\Ne$, we consider the following SDE :
\begin{equation}\label{eq}
X_t =x_0 +\int_0^t \int_\Xi c(s,X_{s^-},u)\tN (ds,du)+\int_0^t \sigma
(s,X_{s^-})dZ_s
\end{equation}
where $x_0\in\R^d$,  $c:\Omega_2\times\R^+
\times \R^d \times \Xi\rightarrow\R^d$ and $\sigma:\Omega_2 \times\R^+ \times \R^d \rightarrow\R^{d\times n}$ are random coefficients which are predictable and satisfy the set of hypotheses below denoted (R).\\
\noindent\underline{Assumption (R)}: For simplicity, we fix all
along this article a finite terminal time $T>0$.
\noindent 1.a) For $\bbP_2$-almost  all $w_2 \in\Omega_2$, all $t\in [0,T]$ and $u\in \Xi$, $c(t,\cdot ,u)$ is
infinitely differentiable  and
\[\forall \alpha\in\Ne, \  \sup_{t\in [0,T], x\in\R^d} | D^\alpha_x c(t,x
,\cdot)| \in \bigcap_{p\geq 1} L^p (\Omega_2 \times \Xi,\bbP_2\times\nu),\] \indent b) $
\sup_{t\in [0,T]}\ |c(t,0,\cdot)|\in \bigcap_{p\geq 1} L^p (\Omega_2 \times \Xi,\bbP_2\times\nu),$ 

\indent c)  for all $t\in [0,T]$, $\alpha\in\N$ and $x\in \R^d$,
$D^\alpha_x c(t,x,\cdot )\in\bbd^\infty$ and
\[  \forall n\in\Ne, \ \forall q\geq 2,\ \sup_{t\in [0,T], x\in\R^d}  \|
D^\alpha_x c(t,x,\cdot)\|_{\bbd^{n,q}}\in \bigcap_{p\geq 1} L^p (\Omega_2 ,\bbP_2),\] \indent d) for all $t\in
[0,T] $, all $x\in \R^d$ and $u\in \Xi$,\ the matrix $I+D_x
c(t,x,u)$ is invertible and
\[ \sup_{t\in [0,T], x\in \R^d} \left|\left(I+D_x c(t,x,u)\right)^{-1}{\times c(t,x,u)}\right|\in \bigcap_{p\geq 1} L^p (\Omega_2 \times \Xi,\bbP_2\times\nu).\] 2.  For all $t\in [0,T]$ , $\sigma(t,\cdot )$ is infinitely
 differentiable and
\[ \forall \alpha\in\Ne \sup_{t\in [0,T], x\in\R^d} | D_x^\alpha \sigma(t,x)|\in \bigcap_{p\geq 1} L^p (\Omega_2 ,\bbP_2 ).\]
3. As a consequence of  hypotheses 1. \!and 2. \!above, it is well
known that equation \eqref{eq} admits a unique solution $X$ such
that $ \E [\sup_{t\in [0,T]} |X_t |^2 ]<+\infty$.
 We suppose that for all $t\in [0,T]$, the matrix $(I+\sum_{j=1}^n D_x \sigma_{\cdot ,j} (t, X_{t^-})\Delta Z_t^j )$
  is invertible and its inverse is bounded by a deterministic constant uniformly with respect to
  $t\in [0,T]$.\\
We shall also consider the following set {of} hypotheses:\\
\noindent\underline{Assumption (\={R})}: assume all the hypotheses of (R) excepted  hypothesis 1.c) which is replaced by the stronger one:\\ 
\indent  1.\={c})  for all $t\in [0,T]$, $\alpha\in\N$ and $x\in \R^d$,
$D^\alpha_x c(t,x,\cdot )\in\bbbd^\infty$ and
\[  \forall n\in\Ne, \ \forall p\geq 1,\ \sup_{t\in [0,T], x\in\R^d}  \|
D^\alpha_x c(t,x,\cdot)\|_{\bbbd^{n,p}}<+\infty .\] 
\subsection{Spaces of processes}
We start by introducing the Dirichlet structure which is the  product   of $(\bbD, \cE)$ and $(\bbd ,e)$ that we denote by $(\bbD\ho\bbd ,\hat{\cE})$. Following \cite{bouleau-hirsch2} Section V.2, we know that 
\begin{eqnarray*}\bbD\ho\bbd&={\big\{ }f\in L^2 (\bbP\times \nu)\makebox{ s.t. } &\makebox{for } \nu  \makebox{ almost all } x\in\Xi,\ f(\cdot ,x )\in \bbD,\makebox{ for } \bbP  \makebox{ almost all } w\in\Omega,\ f(w ,\cdot )\in \bbd\\
&&\makebox{ and } \hat{\cE} (f)=\int_{\Xi} \cE (f(\cdot ,x ))\, d\nu (x)+\E [e (f(w ,\cdot ))]\, <+\infty\}.\end{eqnarray*}
It is a local Dirichlet structure which admits a carré du champ, $\hat{\Gamma }$,  given by 
$$\hat{\Gamma} [f](w ,x )=\Gamma [f(\cdot ,x )](w) +\gamma [f(w ,\cdot)](x),$$
and a Gradient operator, $\hat{D}$, with values in $L^2 (\hbbP )\times L_0 $ given by 
$$ \hat{D}f (w,x)=\left( f^\sharp (\cdot ,x)(w),f^\flat (w,\cdot )(x)\right).$$
Let $\hat{A}$ be the generator of this Dirichlet structure and $\cD (\hat{A})$ its domain. It is obvious  (see Section V, Proposition 2.1.3 in \cite{bouleau-hirsch2})\ that $\cD (\hat{A})$ contains 
$\cD_0 (\hat{A})$ where 
\begin{eqnarray*}\cD_0 (\hat{A})&={\big\{ }f\in L^2 (\bbP\times \nu)\makebox{ s.t.} &\makebox{ for }\nu  \makebox{ almost all } x\in\Xi,\ f(\cdot ,x )\in \cD (A),\\ &&\makebox{ for } \bbP  \makebox{ almost all } w\in\Omega,\ f(w ,\cdot )\in \cD (a), \\
&&\E\left[ \int_{\Xi} |A[f (\cdot ,x )](w)|^2+ |a[f(w ,\cdot )](x)|^2 \, d\nu (x)\right]\, <+\infty\},\end{eqnarray*}
and if $f\in \cD_0 (\hat{A})$, \begin{equation}\label{formuleA}\hat{A}f (w,x)=A[f(\cdot ,x)](w)+a[f(w,\cdot )](x).\end{equation}
We consider the algebraic tensor product $\bbD_0 \otimes \bbd_0$, it is dense in $\bbD\ho\bbd$. Moreover, each element in $\bbD_0 \otimes \bbd_0$ is infinitely differentiable w.r.t. $\hat{D}$, so that we can define as for $\bbD$ (or $\bbd$) the different Sobolev spaces $(\bbD\ho\bbd )^{n,p}$ and  $(\bbD\ho\bbd )^{\infty}=\bigcap_{n\in\Ne , p\geq 1}(\bbD\ho\bbd )^{n,p}$.\\
For all $n\in\Ne$ and $p\geq 1$, we denote by $(\widetilde{\bbD\ho\bbd })^{n,p}$ the completion of the algebraic tensor product $\bbbD^{n,p}\otimes\bbbd^{n,p}$ with respect to the norm
$$\| X\|_{(\widetilde{\bbD\ho\bbd })^{n,p}}=\int_{\Xi} \|X(\cdot ,x)\|_{\bbbD^{n,p}}\, \nu (dx)+\E [ \| X(w,\cdot )\|_{\bbbd^{n,p}}].$$
It is clear that $(\widetilde{\bbD\ho\bbd })^{n,p} \subset \cD_0 (\hat{A})$ and as usual we set 
$$(\widetilde{\bbD\ho\bbd })^{\infty}=\bigcap_{n\in\Ne , p\geq 1}(\widetilde{\bbD\ho\bbd })^{n,p}.$$
We denote by $\cP$ the predictable sigma-field on
$[0,T]\times \Omega$ and we define the following sets of processes:
\begin{itemize}
\item $\LDP{n,p}$ : the space of predictable processes which belong to $L^2 ([0,T]; \bbD^{n,p})$.
\item $\LDdP{n,p}$ : the set of real valued processes $H$ defined on $[0,T]\times
\Omega \times \Xi$ which are predictable and belong to $L^2
([0,T];(\bbD\ho\bbd )^{n,p}).$
\item  $\bLDP{n,p}$ : the set of predictable real valued  processes which belong to $L^2 ([0,T]; \bbbD^{n,p})$.
\item $\bLDdP{n,p}$ : the set of real valued processes $H$ defined on $[0,T]\times
\Omega \times \Xi$ which are predictable and belong to $L^2
([0,T];(\widetilde{\bbD\ho \bbd})^{n,p})$.

\end{itemize}
In a natural way, we set
\[ \LDP{\infty} =\bigcap_{n\in\Ne , p\geq 1} \LDP{n,p},\ \bLDP{\infty} =\bigcap_{n\in\Ne , p\geq 1} \bLDP{n,p}\]
and
\[ \LDdP{\infty} =\bigcap_{n\in\Ne , p\geq 1} \LDdP{n,p},\ \bLDdP{\infty} =\bigcap_{n\in\Ne , p\geq 1}\bLDdP{n,p}.\]
These spaces are endowed with their natural inductive limit topology.

We define $\LDP{\infty}^0$ (resp. $\bLDP{\infty}^0 )$to be the set of elementary processes
in $\LDP{\infty}$ (resp. $\bLDP{\infty}$) of the form
\[ G_t
(w)=\sum_{i=0}^{m-1} F_i (w)\1_{]t_i ,t_{i+1}]}(t),\] where
$m\in\Ne$, $0\leq t_0 \leq\cdots t_m\leq T$ and for all $i$, $F_i \in\bbD^\infty$ (resp. $\bbbD^\infty$) and is $\cA_{t_i}$-measurable.\\
The following Lemma is obvious:
\begin{Le} $\LDP{\infty}^0$ (resp. $\bLDP{\infty}^0 )$ is dense in  $\LDP{\infty}$ (resp. $\bLDP{\infty}$).
\end{Le}
\noindent{\bf Remark:} Let $H\in (\bbD\ho \bbd)^\infty$, then it is infinitely differentiable both w.r.t. to $w\in\Omega$ and $u\in\Xi$. One can easily 
verify (by approximation) that the order of {the derivations} plays no role so that for all $n,k\in\Ne$, the variable $X^{(n\sharp ),(k\flat ) }$ is defined without ambiguity 
as an element in $\bigcap_{p\geq 1} L^p (\bbP\times\hbbP^n \times \nu\times \rho^k )$.
\subsection{Functional calculus related to stochastic integrals}
\begin{Pro}\label{FormuleIS1} Let $H\in \LDdP{\infty}$ then for all $t\in [0,T]$ 
\[ X_t =\int_0^t \int_\Xi H(s,u)\tN (ds ,du)\]
belongs to $\bbD^{\infty}$. \\
And we have:
\begin{eqnarray*}
X_t^\sharp (w,w_1 )&=& \int_0^t\int_\Xi H^\sharp (s,u )(w,w_1 )\tN (ds,du)(w)\\
&&+\int_0^t \int_{\Xi\times R} H^\flat (s,u,r_1 )(w) N\odot\rho (ds,du, dr_1 )(w,w_1 ),
\end{eqnarray*}
\begin{eqnarray*}
X_t^{(2\sharp )} (w,w_1 ,w_2 )&=& \int_0^t\int_\Xi H^{2\sharp } (s,u )(w,w_1, w_2 )\tN (ds,du)(w)\\
&& +\int_0^t \int_{\Xi\times R} H^{\sharp,\flat} (s,u,r_1 )(w, w_1) N\odot\rho (ds,du, dr_1 )(w,w_2 )\\
&& +\int_0^t \int_{\Xi\times R} H^{\sharp,\flat}(s,u,r_1 )(w,w_2) N\odot\rho (ds,du, dr_1 )(w,w_1 )\\
&& +\int_0^t \int_{\Xi\times R^2} H^{(2\flat )} (s,u,r_1, r_2 )(w) N\odot\rho^{\odot 2} (ds,du, dr_1,dr_2 )(w,w_1 ,w_2 ).
\end{eqnarray*}
More generally, for all $n\in \Ne$,
\[X_t^{(n\sharp )}=\sum_{i=1}^{2^n} I_i ,\]
where $I_1=\int_0^t\int_\Xi H^{(n\sharp ) } (s,u )\tN (ds,du)$  and for $i\in \{ 2,\cdots ,2^n\}$, $I_i$ is a term of  the form
\begin{equation*}\begin{split}& I_i (w,w_1 ,\cdots ,w_n ) =\\
&\int_0^t \int_{\Xi\times R^j} H^{(j\sharp),((n-j)\flat )}_{(s,u,r_1 ,\cdots ,r_{n-j})} (w,w_{\sigma (1)}, \cdots ,w_{\sigma (j)})N\odot\rho^{\odot (n-j)} (ds ,du, dr_1 ,\cdots ,dr_{n-j})(w,w_{\sigma (j+1)},\cdots ,w_{\sigma (n )}),\end{split}\end{equation*}
where $j\in \{ 0,\cdots , n-1\}$ and $\sigma$ is a permutation on $\{ 1 ,\cdots ,n \}$.
\end{Pro}
\begin{proof} The case $n=1$  has been established in \cite{bouleau-denis2},  we proceed in a similar way.  
Assume first that 
 \[ H_t
(w,u)=\sum_{i=0}^{m-1} F_i (w)\1_{]t_i ,t_{i+1}]}(t)g_i (u),\] where for all $i\in \{
0,\cdots ,m-1\}$, $F_i \in \bbD^{\infty}$  and is $\cA_{t_i}$-measurable and
$g_i \in \bbd^\infty$.\\
The result is a direct consequence of the functional calculus and  Lemma \ref{DeriveSimple}. We conclude taking first a linear combination and then by density.\end{proof}
\begin{Le}\label{FormuleA}
Let $0\leq s<t\leq T$, $F\in\cD (A)$ and $g\in \cD (a)$.  If $F$ is $\cF_s$-measurable then  
\[ X=F\tN (\1_{]s,t]} g)\]
belongs to $\cD (A) $ and 
\[ A[X]=A[F]\tN (\1_{]s,t]} g)+ F\tN (\1_{]s,t]} a[g]).\]
\end{Le}
\begin{proof}Let us  prove that $X$ belongs to the domain of $A$ and calculate $A[X]$. To this end, assume first that 
\[F=e^{i\tN (f)},\ f\geq 0,\ f\in L^2 ([0,T],dt)\otimes H,\,  \makebox{ and } f(u,\cdot )=0\ \forall u>s\ \quad (**) \]
(the space $H$ is the subspace of $\mathcal{D}(a)$ introduced in hypothesis (H) \S \ref{notations}). 

By the functional calculus (see \cite{bouleau-hirsch2}, Section 1.6), we know that $X$ belongs to $\cD (A)$ and that 
\[ A[X]=A[F]\tN (\1_{]s,t]} g)+ FA[\tN (\1_{]s,t]} g)]+\frac12 \Gamma [F,\tN (\1_{]s,t]} g)],\]
but as a consequence of the explicit expression of $A$ given in \cite{bouleau-denis}, Section 3.2.1, we know that 
\[ A[\tN (\1_{]s,t]} g)]=\tN (\1_{]s,t]} a[g]),\]
and moreover:
\begin{eqnarray*}
\Gamma [F,\tN (\1_{]s,t]} g)]&=&iF\Gamma [\tN (f), \tN (\1_{]s,t]} g)]\\
&=&iFN(\gamma [f, \1_{]s,t]}g,])\\
&=&0,
\end{eqnarray*}
since $\gamma$ acts only on the variable in $\Xi$.\\
Finally, as the space of random variables of the form $(**)$ is total in the subvector space of random variable in $\cD (A)$ and $\cF_s$-measurable  (see Section 3.2.1 in \cite{bouleau-denis}), we conclude by density. 
\end{proof}
\begin{Le}\label{LemmeIS} Let $H\in L^2 ([0,T]; \cD (\hat{A}))$ be predictable, then for all $t\in [0,T]$ 
\[ X_t =\int_0^t \int_\Xi H(s,u)\tN (ds ,du)\]
belongs to $\cD (A)$ and 
\[ A[X_t ] =\int_0^t \int_\Xi \hat{A} [H(s,u)] \, \tN (ds ,du).\]
\end{Le}
\begin{proof} Assume first that $H$ is a ``simple proces'' :
\[ H_t (w,u)=\sum_{i=0}^{m-1} F_i (w)\1_{]t_i ,t_{i+1}]}(t)g_i (u),\] 
where for all $i\in \{ 0,\cdots ,m-1\}$ $F_i\in \cD (A)$ is $\cF_{t_i}$ measurable and $g_i \in \cD (a)$.
Then, for all $t\in [0,T]$
\[ X_t =\sum_{i=0}^{m-1} F_i \tN (\1_{]t_i ,t_{i+1}]}g_i ).\] 
Then, as a consequence of Lemma \ref{FormuleA}:
\begin{eqnarray*}A[ X_t ] &=&\sum_{i=0}^{m-1} \left( A[ F_i ] \tN(\1_{]t_i ,t_{i+1}]}g_i )  + F_i (w)\tN \left(\1_{]t_i ,t_{i+1}]}a[g_i ] \right)\right) \\
&=& \tN\left( \sum_{i=0}^{m-1} \left( A[ F_i ] \1_{]t_i ,t_{i+1}]}g_i   + F_i (w)\1_{]t_i ,t_{i+1}]}a[g_i ] \right)\right).\\
&=&\tN \left( \hat{A} [H]\right)
\end{eqnarray*}
Following \cite{bouleau-hirsch2} (Section V, Proposition 2.1.3), we know that the algebraic tensor product $\cD (A)\otimes\cD (a)$ is dense in $\cD (\hat{A})$ for the graph norm. From this, 
 we deduce that the set of  combinations of simple processes is dense in   $L^2 ([0,T]; \cD (\hat{A}))$ and conclude by density.
\end{proof}
Combining Proposition \ref{FormuleIS1}, Lemma \ref{LemmeIS} and relation \eqref{formuleA} we obtain
\begin{Pro}\label{FormuleIS2} Let $H\in \bLDdP{\infty} $ then for all $t\in [0,T]$ 
\[ X_t =\int_0^t \int_\Xi H(s,u)\tN (ds ,du)\]
belongs to $\bbbD_{\infty}$ and 
\[ A[X_t ] =\int_0^t \int_\Xi \left( A[H(s,u)]+a[H(s,\cdot )](u)\right) \, \tN (ds ,du).\]
\end{Pro}

The proof of the next Proposition is similar to the previous ones and even easier, so we leave it to the reader:
\begin{Pro}\label{FormuleIS3} Let $G\in \LDP{\infty}$ then for all $t\in [0,T]$ 
\[ X_t =\int_0^t \int_\Xi G_s \, dZ_s\]
belongs to $\bbD^{\infty}$, and  for all $n\in \Ne$:
\begin{eqnarray*}
X_t^{(n\sharp )} &=& \int_0^t G_s^{(n\sharp)} \, dZ_s .
\end{eqnarray*}
Moreover, if $G$ belongs to $\bLDP{\infty}$ then $X_t$ belongs to $\bbbD^{\infty}$ and 
\[ A[X_t ] =\int_0^t \int_\Xi A[G_s ]\, dZ_s.\]
\end{Pro}
 
\noindent Finally, by the functional calculus developpped in the proofs of Propositions \ref{FormuleIS1} and \ref{FormuleIS2}, the following Lemma is also clear:
\begin{Le}\label{Compo} Let $c:\Omega_2\times\R^+\times \R^d \times \Xi\rightarrow\R^d$ be the coefficient of equation\eqref{eq} and $X$ be in
$(\LDP{\infty})^d$..
\begin{enumerate}
\item If $c$ satisfies hypothesis 1.c) of (R) then the process $(t,u)\rightarrow c(t,X_t ,u)$ belongs to 
$(\LDdP{\infty} )^d$.
\item If moreover $X$ belongs to $(\bLDP{\infty} )^d$ and $c$ satisfies hypothesis 1.\={c}) of (\={R}) then the process $(t,u)\rightarrow c(t,X_t ,u)$ belongs to 
$(\bLDdP{\infty} )^d$.
\end{enumerate}
\end{Le}
\subsection{Existence of smooth density for the solution}
\begin{Pro}\label{Solution} Under hypotheses (R), the equation \eqref{eq} admits a unique solution, $X$,  in {$(\LDP{\infty})^d$}.
\end{Pro}
\begin{proof} Let us first prove that for all $p\geq 1$, $X$ belongs to $(\LDP{1,p})^d$. We follow the same proof 
as the one of Proposition 8 in \cite{bouleau-denis2} which corresponds to the case $p=2$. \\
 We define inductively a sequence $(X^r)$ of
$\R^d$-valued semimartingales by $X^0 =x$ and
\begin{equation}\label{picard}\forall r\in\N,\ \forall t\in [0,T],\ X^{r+1}_t =x_0 +
\int_0^t \int_\Xi c(s,X^r_{s^-},u)\tN (ds,du)+\int_0^t \sigma
(s,X^r_{s^-})dZ_s.\end{equation}As a consequence of Lemma \ref{Compo}, Propositions
\ref{FormuleIS1} and \ref{FormuleIS3}, it is clear that for all $r$, $X^r$ belongs to
$(\LDP{\infty})^d$ and that we have $\forall t\in [0,T]$
 \begin{eqnarray*}
X_t^{r+1 ,\sharp}& =&\int_0^t\int_U D_x c(s,X^r_{s-},u)\cdot
X^{r,\sharp}_{s-}\tN (ds,du)+\int_0^t\int_{U\times R}
c^{\flat}(s,X^r_{s-},u,r)N\odot\rho (ds,du,dr)\\&&+\int_0^t D_x
\sigma (s,X^r_{s-})\cdot X^{r,\sharp}_{s-} dZ_s.\end{eqnarray*}
This  is the iteration procedure due to \'Emile Picard and it
is well-known that for all $p\geq 1$ \begin{equation}\label{Picard} \lim_{r\rightarrow +\infty}E[\sup_{t\in
[0,T]} |X_t -X^r_t |^p] =0.\end{equation} Moreover, thanks to the hypotheses
we made on the coefficients, it is easily seen (see \cite{bichteler-jacod} or \cite{jacod}) that there exists a
constant {$\kappa_{p,x}$} such that for all $r\in\Ne$ and all $t\in [0,T]$
\begin{eqnarray*}
\E\hE\left[ |X_t^{r+1 ,\sharp}|^p\right] &\leq&{\kappa_{p,x}} \left( 1+\int_0^t \E\hE
\left[|X^{r,\sharp}_{s-}|^p \right]ds \right)\end{eqnarray*} so that by
induction we deduce
\[ \forall r\in\N ,\ \forall t\in [0,T] ,\ \E\hE \left[|X_t^{r ,\sharp}|^p\right]\leq {\kappa_{p,x}} e^{{\kappa_{p,x}} t}.\]
Hence, the sequence $(X^r )$ is bounded in $(\LDP{1,p})^d$ which is a reflexive  Banach space. Therefore, there is a sequence of convex combinations of
$X^r$ which converges to a process $Y\in (\LDP{1,p})^d$. But, by \eqref{Picard} we a priori know that $X^r$ tends to $X$ in $L^p ([0,T];\R^d )$ so that $Y$ is nothing but  $X$. This proves that $X$ belongs to $(\LDP{1,p})^d$. Moreover, still by Propositions \ref{FormuleIS1} and \ref{FormuleIS2} we know that $X^{\sharp}$ satisfies
 \begin{eqnarray*}
X_t^{\sharp}& =&\int_0^t\int_U D_x c(s,X_{s-},u)\cdot
X^{\sharp}_{s-}\tN (ds,du)+\int_0^t\int_{X\times R}
c^{\flat}(s,X_{s-},u,r)N\odot\rho (ds,du,dr)\\&&+\int_0^t D_x
\sigma (s,X_{s-})\cdot X^{\sharp}_{s-} dZ_s.\end{eqnarray*}
This ensures, by standard facts on the Picard's iteration method, that for all $p\geq 1$
\[ \lim_{r\rightarrow +\infty}\E\hE [ \int_0^T | X^{r,\sharp}_s -X^\sharp_s |^p ds ]=0,\]
in other words, we have proved that for all $p\geq 1$, $X^r$ converges to $X$ in $(\LDP{1,p} )^p$.\\
We now proceed by recurrence. Let $n\geq 2$, assume that we have established that $X^r$ converges to $X$ in $(\LDP{(n-1),p})^d$ for all $p\geq 1$. Then, by 
Propositions \ref{FormuleIS1},\, \ref{FormuleIS3} and the functional calculus we have for all $t\in [0,T]$ and all $r\in\Ne$:
\begin{eqnarray*}
X_t^{r+1,(n\sharp )}&=&\int_0^t \int_\Xi D_x c(s,X^r_{s^-}, u)\cdot X_{s^-}^{r,(n\sharp)}\,\tN (ds,du )\\
&&+\sum_{j=1}^n \int_0^t \int_{\Xi\times R^{n-j}}\Phi^j_r (s,u,r_1 ,\cdots, r_{n-j} )\, \tN\odot\rho^{\odot (n-j)} (ds ,du, dr_1,\cdots ,dr_{n-j})\\
&&+ \int_0^t D_x \sigma (s, X^r_{s^-})\cdot X_{s^-}^{r,(n\sharp)}\, dZ_s +\int_0^t \Pi_r (s) \, dZ_s,
\end{eqnarray*}
where:
\begin{itemize}
\item for all $j\in \{ 1,\cdots ,n\}$, $\Phi^j_r $ can be expressed as a sum such that each term is the  product of 
$D_x^k c^{(l\sharp )} (s, X^r_{s^- },u, r_1 ,\cdots ,r_l )$ with $k,l\in \{0,\cdots ,n\}$, and of derivatives of $X$ of order strictly less than $n$;
\item $\Pi_r $ is a sum such that each term is the  product of 
$D_x^k \sigma (s, X^r_{s^- })$ with $k\in \{1,\cdots ,n\}$, and of derivatives of $X$ of order strictly less than $n$.
\end{itemize}
Thanks to the hypotheses we made on $c$ and $\sigma$ and as $X^r$ converges to $X$ in $(\LDP{(n-1),p})^d$, it is clear that there exist 
predictable processes $\Phi^j$, $j=1\cdots n$ and $\Pi$ such that for all $p\geq 1$
\[ \lim_{r\rightarrow +\infty} \E\hE [\int_0^T\int_\Xi \int_{R^{n-j}} | \Phi^j_r -\Phi^j |^p (s,u, r_1 ,\cdots ,r_{n-j})\nu (du)\rho^n (dr_1 \cdots dr_{n-j})]=0,\]
\[\lim_{r\rightarrow +\infty} \E\hE [\int_0^T | \Pi_r -\Pi |^p (s) ds]=0.\]
From this, as in the case $n=1$, we conclude that for all $p\geq 1$,  $(X^{r})_r$ is bounded  in $(\LDP{n,p})^d$ hence a convex combination of $X^r$ converges 
in $(\LDP{n,p})^d$ to an element which is nothing but $X$ and then as $X^{(n\sharp )}$ satisfies a s.d.e., standard consideration on the Picard's iteration permits to conclude that in fact 
\[ \lim_{r\rightarrow +\infty} \E\hE [\int_0^T |X^{r,(n\sharp}_{s^-} -X^{(n\sharp)}_{s^-}|^p ds ]=0,\]
i.e. $(X^r )$ tends to $X$ in $(\LDP{n,p})^d$.
\end{proof}
In the next proposition, for all $i\in\{1,\cdots ,d\}$, we denote by $c_i$ (resp. $\sigma_i$) the $i$-th coordinate of coefficient $c$ (resp. $\sigma$) and we put $X=(X_1 ,\cdots ,X_d )$.
\begin{Pro}\label{FormuleEqA} Under hypotheses (\={R}), $X$ belongs to $\bLDP{\infty}$ and we have for all $t\in [0,T]$ and all $i\in\{ 1,\cdots ,d\}$:
\begin{eqnarray*}
A[X_{i,t} ]&=&\int_0^t \int_X \left(\sum_{j=1}^d\displaystyle\frac{\partial  c_i}{\partial x_j}(s,X_{s^-},u)A[ X_{j,s^-}]+\frac12\sum_{j,k=1}^d\displaystyle\frac{\partial^2 c_i}{\partial x_j\partial x_k }(s,X_{s^-},u)\Gamma [X_{j,s^-},X_{k,s^-}]\right)\, \tN (ds,du)\\&&
+\int_0^t \int_\Xi a[c_i (s,X_{s^-},\cdot )](u)\, \tN (ds ,du)\\
&&+\int_0^t \left( \sum_{j=1}^d\displaystyle\frac{\partial  \sigma_i}{\partial x_j}(s,X_{s^-})A[ X_{j,s^-}]+\frac12\sum_{j,k=1}^d\displaystyle\frac{\partial^2 \sigma_i}{\partial x_j\partial x_k }(s,X_{s^-})\Gamma [X_{j,s^-},X_{k,s^-}]\right)\, dZ_s
\end{eqnarray*}
\end{Pro}
\begin{proof} We keep the same notations as in the prof of the previous Proposition, so we still consider $(X^r )$ the Picard's approximation of $X$ given by 
relation \eqref{Picard}. By Proposition \ref{FormuleIS3}, Lemma \ref{Compo} we know that for all $r\in\Ne$, $X^r$ belongs to 
$(\bLDP{\infty})^d$ and that moreover:
\begin{eqnarray*}
A[X^{r+1}_{t} ]&=&\int_0^t \int_X\left( A[ c(s,X^r_{s^-},u)]+a[c(s,X_{s^-}^r ,\cdot )](u)\right)\, \tN (ds ,du)\\
&&+\int_0^t  A[ \sigma(s,X^r_{s^-})]\, dZ_s .
\end{eqnarray*}
And, by the functional calculus (see Corrolary 6.1.4. in \cite{bouleau-hirsch2}), this yields for all $i\in \{ 1,\cdots ,d\}$ (we apologize for the apparently complicated computation  but in fact quite natural):
\begin{eqnarray*}
A[X^{r+1}_{i,t} ]&=&\int_0^t \int_\Xi \left(\sum_{j=1}^d\displaystyle\frac{\partial  c_i}{\partial x_j}(s,X^r_{s^-},u)A[ X^r_{j,s^-}]+\frac12\sum_{j,k=1}^d\displaystyle\frac{\partial^2 c_i}{\partial x_j\partial x_k }(s,X^r_{s^-},u)\Gamma [X^r_{j,s^-},X^r_{k,s^-}]\right)\, \tN (ds,du)\\&&
+\int_0^t \int_\Xi a[c_i (s,X^r_{s^-},\cdot )](u)\, \tN (ds ,du)\\
&&+\int_0^t \left(\sum_{j=1}^d\displaystyle\frac{\partial  \sigma_i}{\partial x_j}(s,X^r_{s^-})A[ X^r_{j,s^-}]+\frac12\sum_{j,k=1}^d\displaystyle\frac{\partial^2 \sigma_i}{\partial x_j\partial x_k }(s,X^r_{s^-})\Gamma [X^r_{j,s^-},X^r_{k,s^-}]\right)\, dZ_s
\end{eqnarray*}
Let us now introduce $Y=(Y_1 ,\cdots, Y_d )$ the solution of the following s.d.e:
\begin{eqnarray*}
Y_{i,t}&=&\int_0^t \int_\Xi \left(\sum_{j=1}^d\displaystyle\frac{\partial  c_i}{\partial x_j}(s,X_{s^-},u)Y_ {j,t} +\frac12\sum_{j,k=1}^d\displaystyle\frac{\partial^2 c_i}{\partial x_j\partial x_k }(s,X_{s^-},u)\Gamma [X_{j,s^-},X_{k,s^-}]\right)\, \tN (ds,du)\\&&
+\int_0^t \int_\Xi a[c_i (s,X_{s^-},\cdot )](u)\, \tN (ds ,du)\\
&&+\int_0^t \left(\sum_{j=1}^d\displaystyle\frac{\partial  \sigma_i}{\partial x_j}(s,X_{s^-})Y_ {j,t} +\frac12\sum_{j,k=1}^d\displaystyle\frac{\partial^2 \sigma_i}{\partial x_j\partial x_k }(s,X_{s^-})\Gamma [X_{j,s^-},X_{k,s^-}]\right)\, dZ_s
\end{eqnarray*}
In the previous Proposition, we have proved that $X^r$ converges to $X$ in $\LDP{\infty}$ so for all $p\geq 1$:
\[ \lim_{r\rightarrow +\infty}\E [ \int_0^T | \Gamma [X_s]-\Gamma [X^r_s ]|^p \, ds =0.\]
From this, it is standard to prove that $A[ X^r ]$ converges to $Y$ in $L^p ([0,T]\times \Omega )$.\\
This implies that $X^r$ is a Cauchy sequence in $\bLDP{n,p}$ for all $(n,p)$ hence it converges to 
a limit which is $X$  and necessary, $Y=A[X]$. This ends the proof
\end{proof}
We are now able to give the main Theorem of this section. {For this we need} some processes introduced in \cite{bouleau-denis2}.\\
 First, the $\R^{d\times d}$-valued process $U_s$ defined by
$$dU_s=\sum_{j=1}^nD_x\sigma_{.,j}(s,X_{s-})dZ_s^j.$$
Then the  $\R^{d\times d}$-valued
process which is the derivative of the flow generated
by $X$:
\begin{eqnarray*}
K_t &=& I+\int_0^t\int_\Xi D_x c(s,X_{s-} ,u)K_{s-} \tN (ds ,du)
+\int_0^t dU_sK_{s-}
\end{eqnarray*}
\noindent Under our
hypotheses, for all $t\geq 0$, the matrix $K_t$ is invertible and its inverse   $\bK_t =(K_t )^{-1}$ satisfies:
\begin{eqnarray*}
\bK_t &=& I -\int_0^t\int_\Xi \bK_{s-} (I+D_x c(s,X_{s-} ,u))^{-1} {
D_x c(s,X_{s-} ,u)}\tN (ds ,du)
\\&&-\int_0^t \bK_{s-} dU_s+\sum_{s\leq t}\bK_{s-}(\Delta U_s)^2(I+\Delta U_s)^{-1}+\int_0^t\bK_s d<U^c,U^c>_s.
\end{eqnarray*}
The key property is  that under hypotheses (R), we have the following relation (see \cite{bouleau-denis2}, Theorem 10) for all $t\in [0,T]$,
\begin{eqnarray*}
\Gamma [X_t ]&=&K_t  \int_0^t \int_\Xi\bK_{s} \gamma[c(s
,X_{s-} ,\cdot )]\bK_{s}^{\ast}\, N (ds, du)
K_t^{\ast},
\end{eqnarray*}
where for any matrix $M$, $M^{\ast}$ denotes its transposed. 
\begin{Th}\label{Thregul} Assume hypotheses (\={R}). Let $t>0$,  if 
\[ \left(  \int_0^t \int_\Xi\bK_{s} \gamma[c(s
,X_{s-} ,\cdot )]\bK_{s}^{\ast}\, N (ds, du)\right)^{-1}\in\bigcap_{p\geq 1} L^p (\bbP, \R^{d\times d} )\]
then $X_t$  admits a density which 
belongs to $C^{\infty} (\R^d )$.
\end{Th}
\begin{proof} Let $t\in [0,T]$. The idea is to apply Proposition \ref{Densite} to $X_t$. Clearly, it remains to prove that 
$(\Gamma [X] )^{-1}$ belongs to $\bigcap_{p\geq 1} L^p (\bbP, \R^{d\times d} )$. But  this is obvious because 
$K_t$ belongs to $\bigcap_{p\geq 1} L^p (\bbP, \R^{d\times d} )$.
\end{proof}
\noindent {\bf Remark:} Let us {emphasize that the L\'evy measure is not assumed to possess a density, so that} the regularity of density for the solution is obtained under weaker hypotheses than those of  L\'eandre (cf \cite{leandre1} \cite{leandre2}) or other authors (see \cite{bichteler-gravereaux-jacod}).\\
Moreover, in these works and in some other works dealing with {non necessarily absolutely continuous L\'evy measure, a growth condition is supposed} near the origin of the form:
$$\liminf_{\rho\rightarrow 0} \rho^{-\alpha}\int_{\{|x|<\rho\}} |x|^2 d\nu (x)>0,$$
with $\rho\in (0,2)$, that we do not assume.
\section{Applications}
\subsection{The regular case}
As in Proposition 11 in \cite{bouleau-denis2}, we assume that  $\Xi$ is a topological space and that coefficient $c(s,x,u)$ is regular with respect to its argument $u$ governing the jumps size, this ensures existence of the density. The regularity of the density will be the consequence of an elliptic-type assumption on $c$. More precisely we have 
 \begin{Pro}{\label{regularcase} }Assume hypotheses (\={R}), that $\Xi$ is a topological space and  that the intensity measure $ds\times\nu$ of $N$ is such that $\nu$ has an infinite mass near some point $u_0$ in $\Xi$. Assume that the matrix $(s,x,u)\rightarrow\gamma[c(s,x,\cdot)](u)$ is continuous on a neighborhood of $(0,x_0,u_0)$ and  invertible at $(0,x_0,u_0)$. Assume moreover that it satisfies the following (local) ellipticity assumption:  \[ \forall (s',x,u)\in ]0,s]\times\R^d \times \cO,\ \gamma [ c(s' ,x,u]\geq \displaystyle\frac{1}{1+|x|^\delta}\psi (u) I_d ,\]
Where  $\geq$ denotes the order relation in the set of symmetric and positive matrixes,  $\delta , s>0$ are constant, $I_d$ is the identity matrix in $\R^{d\times d}$, $\cO$ is a neighborhood of $u_0$ and $\psi$ is an $\R^+\setminus \{ 0\}$-valued measurable function on $\cO$ such that 
 $$\left( \int_0^t \int_{\cO} \psi (u)\, N(ds ,du)\right)^{-1}\in \bigcap_{p\geq 1} L^p (\bbP) . \ \ \ (*)$$\\
Then, for all $t\geq s$  the solution $X_t$ of {\rm(\ref{eq})} admits a density in $C^{\infty} (\R^d )$.
 \end{Pro}
 \begin{proof} As a consequence of Proposition 11 in \cite{bouleau-denis2}, it just remains to prove that $(\Gamma [X_t ])^{-1}$ belongs to 
 $\bigcap_{p\geq 1} L^p (\bbP, \R^{d\times d} )$. \\
 We have for any vector $v\in\R^d$:
\begin{eqnarray*}
v^\ast \Gamma [X_t ]v&\geq& \int_0^t\int_{\cO} v^\ast \bK_{s} \gamma[c(s
,X_{s-} ,\cdot )]\bK_{s}^{\ast} v\, N (ds, du)\\
&\geq & \int_0^t\int_{\cO} v^\ast \bK_{s}\bK_{s}^{\ast} v \displaystyle\frac{1}{1+|X_{s^-}|^\delta}\psi (u) \, N (ds, du).
\end{eqnarray*}
As it is well known that both $\sup_{s\in [0,T]} |K_s|$ and $\sup_{s\in [0,T]}|X_{s^-}|$ belongs to $\bigcap_{p\geq 1} L^p (\bbP)$ we deduce that there exists 
a random variable $V$ such that  $V^{-1}\in \bigcap_{p\geq 1} L^p (\bbP )$ with
$$ \Gamma [X_t ]\geq V \int_0^t \int_{\cO} \psi (u) \, N(ds,du) I_d .$$
It is now easy to conclude.
\end{proof}
We now give a criterion which ensures that $(*)$ is satisfied:
\begin{Le}\label{Tauber} Consider $\cO$ and $\psi$ as above and assume that there exists $\alpha\in (0,1 )$ such that the limit 
$$ r_1=\lim_{\lambda \rightarrow  +\infty} \displaystyle\frac{1}{\lambda^\alpha} {\int_{\cO} (e^{-\lambda \psi (u)}-1)\, \nu (du)}$$
exits and belongs to  $(-\infty , 0)$ then hypothesis $(*)$ of the previous proposition is fulfilled.
\end{Le}
\begin{proof} Set $V=\int_0^t \int_{\cO} \psi (u)\, N(ds ,du)$. Let us first remark that as $\nu (\cO )=+\infty$ and $\psi >0$, $\mathbb{P}(V=0 )=0$. 
The Laplace transform of $V$ is given by
$$ \forall \lambda \geq 0 ,\ \E [e^{-\lambda V}]= {e^{t\int_{\cO} (e^{-\lambda \psi (u)}-1)\, \nu (du)}}.$$
By the De Bruijn's Tauberian Theorem (see Theorem 4.12.9 in \cite{bingham}), we know that this implies 
$$ \lim_{\varepsilon \rightarrow 0} \varepsilon^{\beta}\log ( \mathbb{P}(V\leq \varepsilon ))=r_2,$$
where $\beta$ and $r_2$ satisfy:
$$ \frac{1}{\alpha}=\frac{1}{\beta}+1 \makebox{ and } |\alpha tr_1|^{1/\alpha}=|\beta t r_2 |^{1/\beta }.$$
In other words:
$$\mathbb{P}(V^{-1}\geq x)\underset{ x\rightarrow +\infty}{\sim}e^{\frac{r_2}{x^{\beta}}}.$$
This leads to the result.
\end{proof}
\subsection{Non linear subordination}

We now turn out to an example for which $\Xi$ is infinite dimensional. More precisely, we put\\
\begin{itemize}
\item $\Xi=\R^+\times C_0(\R^+  ;\R^q)$ where $q\in\Ne$ and $W=C_0(\R^+  ;\R^q)$ denotes the set of $\R^q$-valued continuous functions defined on $\R^+$ and vanishing in $0$.
\item $\nu=\tau\times m$ where $m$ is the Wiener measure on $C_0(\R^+  ;\R^q)$ and $\tau$ is a Lévy measure on $\R^+$  associated to a subordinator such that $\tau (\R^+ )=+\infty$. For simplicity, we assume that the support of $\tau$ is included in $[0,T']$ for some $T'>0$. Let us recall that necessarily  $\int_0^{T'} y\tau (dy)<+\infty$.
\item The Dirichlet structure $(\bbd ,e,\gamma)$ is the product of the trivial structure on $L^2 (R^+ ,\tau )$ with  $(\bbd_M ,e_M,\gamma_M)$, the Dirichlet structure on $L^2 (C_0(\R^+  ;\R^q),m)$
associated with the Ornstein-Uhlenbeck operator (see \cite{bouleau-hirsch2}) .
\end{itemize}
As usual,  we denote by $(B_t )_{t\geq 0}$ the coordinates maps on $W$:
$$\forall \omega\in W ,\ B_t (\omega )=\omega_t ,$$
 so that $(B_t)_{t\geq 0}$ is a $q$-dimensional Brownian motion under the probability $m$.\\

 \subsubsection{A very simple example}
Let us now give a basic example which proves that in the case of a non-linear subordination, even if a diffusion  is degenerated, the "subordinated" process $X$ may have a smooth density. We keep the notations above and take $d=q=2$. 
Consider for all $t\geq 0$:
$$X_t =\left(\begin{array}{c}
\int_0^t\int_0^{T'}\int_W B^1_y (\omega) N(ds,dy,d\omega )\\
\frac12  \int_0^t\int_0^{T'}\int_W (B^1_y (\omega) )^2 N(ds,dy,d\omega )\end{array}\right),$$
where $B^1$ denotes the first coordinate of the $2$-dimensional Brownian motion $B$.\\
{ As explained in the remark at the end of this subsection, the process $X$ may be viewed as a {\it non-linear} subordination of the $2$-dimensional diffusion $\zeta=(B^1 ,B^1 )$.}\\ 
 By the functional calculus related to stochastic   integrals, we have for all $t\geq 0$:
 $$\Gamma [X_t ]=\int_0^t \int_\Xi \gamma_M [ B_y^1 ,\frac12 (B_y^1 )^2 ](\omega) N(ds,dy,d\omega ).$$
 It is standard that 
 $$\gamma_M [ B_y^1 ,\frac12 (B_y^1 )^2 ]=\left(\begin{array}{cc}y & yB^1_y \\yB^1_y & y(B^1_y )^2\end{array}\right).$$
Let us now study the matrix $\Gamma [X_t ]$ in order to give a criterion which ensures that $X_t$ admits a smooth density. \\
Let $u=\left(\begin{array}{c}u_1\\u_2\end{array}\right)$ be in $\R^2\setminus\{0\}$. We have
\begin{eqnarray*}
u^*\cdot\Gamma [X_t] \cdot u&=&\int_0^t \int_{\Xi} y(u_1 +B_y u_2 )^2\, N(ds,dy,d\omega).\end{eqnarray*}
By remarking that if $u_1 u_2 \geq 0$, $y<1$ and $B_y \geq \sqrt{y}$ or  if $u_1 u_2 \leq 0$, $y<1$ and  $B_y \leq -\sqrt{y}$ then $(u_1 +B_y u_2 )^2 \geq y(u_1^2 +u_2^2 )$, we deduce easily that 
$$ \Gamma [X_t ]\geq (M_1 \wedge M_2)I,$$
where
$$ M_1 =\int_0^t \int_{\Xi} {\bf{1}}_{\{ y<1\, ,\, B_y \geq \sqrt{y}\}}y^2\,  N(ds,dy,d\omega)\makebox{ and }M_2=\int_0^t \int_{\Xi} {\bf{1}}_{\{ y<1\, ,\, B_y \leq -\sqrt{y}\}}y^2\,  N(ds,dy,d\omega).$$
As clearly $M_1$ and $M_2$ have same law, we only study $M_1$.\\
We first remark that 
$$\int_0^t \int_0^{T'\wedge 1} P(B_y\geq \sqrt{y})\tau(dy)dt =tP(B_1 \geq 1)\int_0^{T'\wedge 1}\tau (dy)=+\infty .$$
Hence, almost surely between times $0$ and $t$, there are infinitely many jumps whose "mark" belongs to the set $\{ (y,w)\in\Xi;\, y\leq 1 ,\, B_y (w)\geq\sqrt{y}\}$. From this, we deduce that $M_1 >0$ almost-surely and that $\Gamma [X_t]$ is invertible which ensures that $X_t$ admits a density.\\
Moreover, as a consequence of Theorem \ref{Thregul} we know that if 
$$M_1^{-1}\in L^p (\Omega) \  \forall p\geq 1$$
then $X_t $ admits a density which belongs to $C^{\infty}_b (\R^2 )$.\\
Here again, as in Lemma \ref{Tauber}, we calculate the Laplace transform of $M_1$:
$$\forall \lambda >0,\  E[e^{-\lambda M_1}]=e^{tP(B_1 \geq 1)\int_0^{T'\wedge 1} (e^{-\lambda y^2}-1)\, \tau (dy)}.$$
This yields:
\begin{Le} Assume that there exists $\alpha\in (0,1 )$ such that the limit 
$$ r_1=\lim_{\lambda \rightarrow  +\infty} \displaystyle\frac{1}{\lambda^\alpha} {\int_0^{T'\wedge 1} (e^{-\lambda y^2}-1)\, \tau (dy)}$$
exits and belongs to  $(-\infty , 0)$ then $X_t $ admits a density which belongs to $C^{\infty}_b (\R^2 )$.
\end{Le}
\noindent{\bf Remark:} The hypothesis of the Lemma is fulfilled if for example 
$\tau (dy)=\frac{1}{y^{1+\epsilon}} dy $ 
with $\epsilon\in (0,1)$.

\subsubsection{Non-linear subordination of a diffusion}
We can generalize the previous example and consider that $X$ is the solution of the following equation
$$
X_t =x_0+\int_0^t \int_\Xi c(s,X_{s^-},y,\omega)\tN (ds,dy,d\omega )+\int_0^t \sigma
(s,X_{s^-})dZ_s
$$
where the semimartingale $Z$ and the coefficient $\sigma$ satisfy the same assumptions as in the previous Section.
The coefficient $c$ is defined thanks to a diffusion process $\zeta$ and is of the form
$$c(t,x,y,\omega )=F(t,x,y, \zeta_y^x -x).$$ Let us now give more details: we first consider the diffusion process $\zeta$ depending on the parameter $x\in\mathbb{R}^d$  and solution of the following SDE:
\begin{equation}\label{eqzeta}\zeta^{x}_t=x+\int_0^ta(\zeta^{x}_s)\,dB_s(w)+\int_0^tb(\zeta^{x}_s)\,ds,\end{equation}
where $a:\, \R^{d\times q}\rightarrow \R^d$ is assumed to belong to $C_K^{\infty} (\R^{d\times q})$, the set of  infinitely differentiable functions with compact support, and $b:\, \R^d \rightarrow \R^d$ is assumed to be bounded, infinitely differentiable with bounded partial derivatives of all orders.\\
Let us remark that, as shown by the basic example above,  these assumptions are not necessary, nevertheless they ensure that hypotheses (R) are fulfilled as we shall see now. \\
It is well known (see \cite{kunita}, \cite{bouleau-hirsch2} Section IV.4 or \cite{denis2}) that we can choose a version of $\zeta$ that we still denote $\zeta$ such that
\begin{enumerate}
\item For almost all $\omega \in W$, for all $t\geq 0$, $x\longrightarrow \zeta_t^x (\omega )$ is a $C^{\infty}$-diffeomorphism from $\R^d$ onto $\R^d$.
\item For almost all $\omega \in W$, for all $\alpha \in \N$, the map $(t,x)\in \R^+ \times \R^d \longrightarrow D^\alpha_x \zeta_t^x (\omega )$ is continuous.
\item For all $t\geq 0$ and all $x\in\R^d$, $\zeta_t^x$ belongs to $(\bbd_M^\infty )^d$.
\item The map $(t,x)\in \R^+ \times \R^d \longrightarrow \zeta_t^x \in (\bbd_M^\infty )^d$ is continuous.
\end{enumerate}
Moreover, as $a$ belongs to $C_K^{\infty} (\R^{d\times q})$ and $b$ is bounded, it is clear that there exists a constant $C>0$ such that
$$\sup_{t\in [0,T], x\in\R^d}| \xi_t^x -x|\leq C \ \makebox{a.e.}$$
Even more, there exists a constant $R>0$ such that if $|x|\geq R$, then $\zeta_t^x$ is deterministic for all $t\in [0,T']$ and satisfies 
\begin{equation}\label{deterministic}\forall t\in [0,T'],\ \zeta_t^x =x +\int_0^t b(\zeta_s^x) \, ds.\end{equation}
For all $t\geq 0$ and all $x\in \R^d$ we denote $\gamma_{M,t}^x =\gamma_M [\zeta_t^x]$, the Malliavin carré du champ of the random variable $\zeta_t^x$.   
We know (see \cite{ikeda-watanabe} or \cite{bouleau-hirsch2}) that
\begin{equation}{\label{eqgammaM}}
\forall t\geq 0 \, \forall x\in\R^d ,\ \gamma_{M,t}^x=M_t \int_0^t M^{-1}_t a(\zeta_t^x )a^*(\zeta_t^x )(M_s^{-1})^*\, dsM_t^*,
\end{equation}
where $(M_t )_{t\geq 0}$ and $(M_t^{-1} )_{t\geq 0}$ satisfy the following SDE's:
$$M_t =I+\sum_{j=1}^q \int_0^t a'_{\cdot ,j} (\zeta_s^x)M_s \, dB^j_s +\int_0^t b' (\zeta_s^x )M_s\, ds$$
and
$$M_t^{-1} =I-\sum_{j=1}^q \int_0^t M_s^{-1} a'_{\cdot ,j} (\zeta_s^x) \, dB^j_s +\int_0^t\sum_{j=1}^q \left((M_s^{-1} a'_{\cdot ,j} (\zeta_s^x ))^2-M_s^{-1}b' (\zeta_s^x )\right)\, ds.$$
Here, $a'_{\cdot ,j}$ denotes the Jacobian matrix of $a_{\cdot ,j}$, the $j$-th column of $a$ and $b'$ the Jacobian matrix of $b$.\\
For $n\in\N$ and $p\geq 1$, in a natural way we denote by $\bbd_M^{n,p}$ the usual Sobolev space of the Malliavin calculus: the space of elements in $L^p (W,m)$ which are $n$ times differentiable in the sense of Malliavin with $p$-integrable derivatives. Let $\alpha \in\N$, by studying the SDE satisfied by $D^{\alpha}_x \zeta^x$, it is quite standard to get the following estimates (see \cite{denis2} for example):
$$\forall s,t\geq 0 ,\, \forall x,x'\in\R^d ,\ \parallel D^\alpha_x \zeta^x_t -D^\alpha_x \zeta^{x'}_s \parallel_{\bbd^{n,p}_M} \leq C_{n,p}\left(1+|x|^n\right)\left(|x-x'|+|t-s|^{1/2}\right),$$
where $C_{n,p}$ is a constant.\\
We set 
$$ \pi (t,x)=\zeta^x_t -x .$$
Applying the Kolmogorov criterion w.r.t. the variable $x$ to the malliavin derivatives of $\zeta_t^x$, we get from this the following estimate:
$$\forall t\geq 0 ,\ \sup_{|x|\leq R} \| D^\alpha_x\pi (t,x)\|_{\bbd^{n,p}_M}\leq C'_{n,p} t^{1/2},$$
where $C'_{n,p}$ is a constant depending on $n,p,R$. Since for $|x|\geq R$, $\zeta^x$ satisfies \eqref{deterministic}. So that, we have the following estimate:
$$\forall t\geq 0 ,\ \sup_{x\in\R} \| D^\alpha_x\pi (t,x)\|_{\bbd^{n,p}_M}\leq C''_{n,p} t^{1/2},$$
where $C''_{n,p}$ is another constant. This yields
\begin{equation}\label{estimezeta}\forall \alpha\in\N,\, \forall n\in\N ,\, \forall p\geq 1 ,\ \int_0^{T'}  \sup_{x\in\R} \| D^\alpha_x\pi (t,x)\|_{\bbd^{n,p}_M}^p \, \tau (dt)<+\infty.\end{equation}
Consider now a function 
$$\begin{array}{cccc}
F: &\R^+ \times \R^d\times [0,T'] \times \R^d &\longrightarrow& \R^d\\
&(t,x,y,z)&\longmapsto& F(t,x,y,z)\end{array}$$
and we assume that for all $t,x,y$, $F(t,x,y,0)=0$, that for all $t,y\geq 0$, $(x,z)\rightarrow F(t,x,y,z)$ is infinitely differentiable and 
$$\forall \alpha \in\N,\ \forall \beta\in\N^* ,\ \sup_{t,x,y,z}| D_x^\alpha D_z^\beta F(t,x,y,z)|<+\infty.$$
We put
$$ \forall (t,x,y,\omega)\in\R^+\times \R^d \times \R^+\times W ,\ c(t,x,y,w)=F(t, x,y,\zeta^{x}_{y} (\omega )-x).$$
As a consequence of the estimate \eqref{estimezeta} and of the boundedness of the the derivatives of $F$, it is clear that $c$ satisfies conditions 1.a), 1.b) and 1.c) of the set of hypotheses (R).\\
It remains to verify hypothesis 1.d). We have for all $(y,\omega)\in\Xi$, all $t\in [0,T]$ and all $x\in\R^d$:
$$ I+D_x c(t,x,y,\omega )=I+D_x  F(t,x,y,\pi^x_y (w)) +D_z F(t,x,y,\pi^x_y (w))\cdot D_x \pi_y^x (w),$$
from this we can check if hypothesis 1.d) is satisfied on concrete examples.\\
Then we are able to make all the calculations:
\begin{eqnarray*}
\Gamma [X_t ]&=&K_t  \int_0^t \int_\Xi\bK_{s} \gamma_M[F(t,X_{s^-},y,\zeta^{X_{s^-}}_y -X_{s^-}) ](w)\bK_{s}^{\ast}\, N (ds, dy,d\omega)
K_t^{\ast}\end{eqnarray*}
and we know that for all $x\in\R^d$:
\begin{eqnarray*}
\gamma_M[F(t,x,y,\zeta^{x}_y -x) ](w)&=&F'(t,x,y,\zeta^{x}_y -x )\times\gamma_M [\zeta^{x}_y ]\times[F'(t,x,y,\zeta^{x}_y -x )]^*,
\end{eqnarray*}
and $\gamma_M [\zeta^{x}_y ]$ is given by equation \eqref{eqgammaM}.\\

{\bf{Remarks:}}\\
1- To understand why we call this example ``non linear subordination'', consider the following case. Take for $c$
$$c(t,x,y,w)=\zeta^{x}_{y} (\omega )-x,$$
and $\sigma=0$. Then 
$$\forall t\geq 0,\ X_t =x_0 +\int_0^t \int_{\Xi}(\zeta^{X_{s^-}}_{y} (\omega )-X_{s^-})N (ds, dy,d\omega ).$$
One can easily verify that in this case, the law of $X_t$ is the law of the diffusion $\zeta$ starting from $x_0$ and subordinated by a subordinator whose Lévy measure is $\tau$. In other words, $(X_t )_{t\geq 0}$ has the same law as $(\zeta^{x_0}_{Y_t})_{t\geq 0}$ where $Y$ is a subordinator independent of $B$ whose Lévy measure is $\tau$.\\
2- The example given first corresponds to the case where the coefficient $\sigma$ is constant and equal to $\sigma =\left(\begin{array}{c} 0\\ \frac12 \int_0^{T'}y\tau (dy)\end{array}\right)$, $dZ_t = dt$ and $$ F(t,x,y,z)=\left(\begin{array}{c}
 z_1 \\
\frac12 z_2^2
\end{array}\right).$$
3- The case we consider here is not the most general possible, indeed we can deal with more sophisticated examples without any difficulty. For example, in equation $\eqref{eqzeta}$ we can consider the case where 
coefficients depend on time $s$ and even by considering one more time a product Dirichlet structure, we can replace the Lebesgue measure $``ds"$ by $``d\tilde{Z}_s"$ where $\tilde{Z}$ is a continuous semimartingale independent of $B$ and satisfying good integrability conditions. We can also consider the case where $F$ is random.\\

\subsection{ Diffusive particle subjected to a L\'evy field of force}
We end by a more sophisticated example which, in some sense, may be viewed as a generalization of the previous one. The aim of this last subsection is to give an idea of what our technic could bring to this kind of example, so that we do not give all the details. \\
Let  $\Xi=\mathbb{R}^2\times\mathcal{C}(\mathbb{R}_+,\mathbb{R}^2)$ equipped with the measure $\nu\times m$ where  $\nu$ is a L\'evy measure and $m$ the Wiener measure with starting point zero.

Let us consider a Poisson random measure  $N(dt,du)$ on $\mathbb{R}_+\times \Xi$ with intensity measure $dt\times\nu\times m$.

A current point of $\mathbb{R}^2$ is represented by $re^{i\theta}$. A Dirichlet form is put on the argument $\theta$ and the Ornstein-Uhlenbeck form is put on the Brownian motion what defines the bottom space by product.

We consider the following SDE 
\begin{equation}\label{sde A1}dX_t=c(X_{t-},u)\;\tilde{N}(dt,du)\qquad X_0=x_0\in \mathbb{R}^2\end{equation} where the function  $c(x,u)$ is defined as follows:

We are given a field $\upsilon(x)$ of $2\times2$ matrices. For $u=(re^{i\theta},\omega)\in \Xi$  we consider the SDE
\begin{equation}\label{sde A2}Z_s^x(\theta)=\int_0^s\upsilon(Z_\tau^x(\theta)+x)\,dB_\tau(\omega)+s\left(\begin{array}{c}
\cos\theta\\
\sin\theta
\end{array}
\right)\end{equation}
during the time $r$, and we put
$c(x,u)=Z^x_{r}(\theta).$ In other words
$$c(X_{t-},u)=Z^{X_{t-}}_{r}(\theta).$$
Let us remark that the process $X_t$ in the case $\upsilon\equiv 0$ is the centered L\'evy process associated with the L\'evy measure $\nu$ :
$$Y_t=Y_0+\int_0^t re^{i\theta}\tilde{N}(ds,du)$$ whose jumps give the drift in equation (\ref{sde A2}). Thus the process $X_t$ may be seen as modeling a particle diffusing with matrix $\upsilon$ subjected to a varying field of forces given by the L\'evy process $Y_t$.

\begin{center}
\includegraphics[width=12cm]{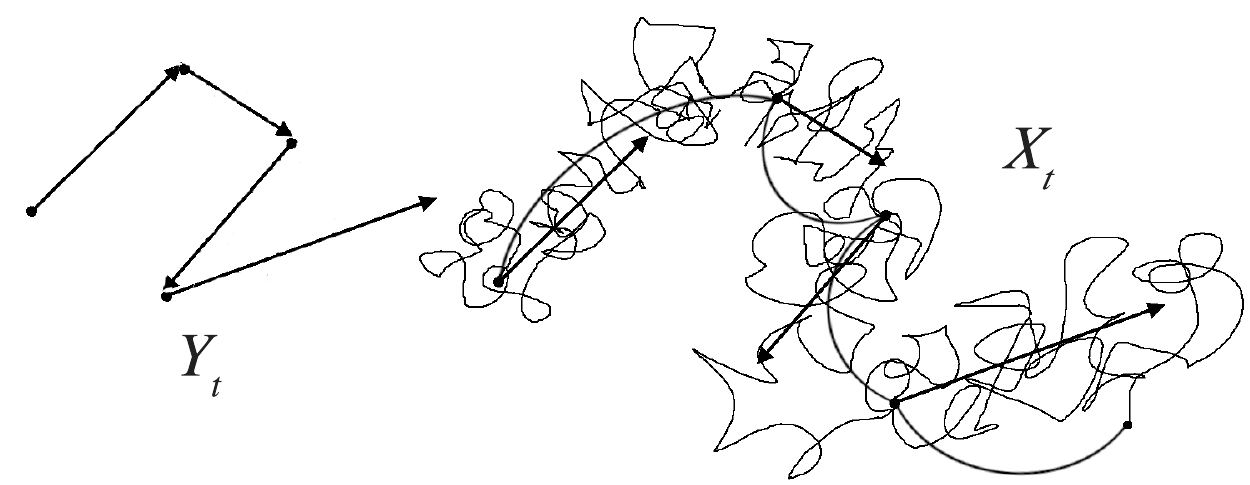}\\
figure 1: The jumps of the particle $X_t$ are the paths of a diffusion particle whose drift is governed by a L\'evy process.\\
\end{center}
Using the results of Bismut  on stochastic flows (see \cite{bismut}) and similarly to the previous example, putting strong regularity assumptions on $\upsilon$ we can 
verify that hypotheses (R) are satisfied and calculate the matrix $\Gamma [X_t]$ in order to apply our criteria.

 Ecole des Ponts,\\
ParisTech, Paris-Est\\
6 Avenue Blaise Pascal\\ 77455 Marne-La-Vallée Cedex 2
FRANCE\\bouleau@enpc.fr \\  \\ Equipe Analyse et Probabilités,
\\Universit\'{e} d'Evry-Val-d'Essonne,\\Boulevard François Mitterrand\\
91025 EVRY Cedex FRANCE\\ldenis@univ-evry.fr

\begin{thebibliography}{00}
\bibitem{ane-et-al}{\sc An\'e C. et al.} "Sur les in\'egalit\'es de Sobolev logarithmiques" Panoramas et Synth\`eses, SMF (2000).
\bibitem{bakry}{\sc Bakry D.} "Transformations de Riesz pour les semigroupes sym\'etriques" S\'em Proba XIX, LNM 1123 Springer (1985).
\bibitem{bakry-emery}{\sc Bakry D., Emery M.} "Diffusions hypercontractives" {S\'em. Strasbourg XIX} p177 Springer (1985)
\bibitem{bichteler-gravereaux-jacod}{\sc Bichteler K., Gravereaux J.-B., Jacod J.} {\it Malliavin Calculus for Processes with Jumps} (1987).
\bibitem{bichteler-jacod} {\sc K. Bichteler, J. Jacod}  Calcul de Malliavin pour les diffusions avec sauts, existence
d'une densit\'e dans le cas uni-dimensionnel, in: {\em S\'eminaire de Probabilit\'es XVII}, Lect. Notes in Math. 986,
132-157, Springer Verlag, (1983).
\bibitem{bingham}{\sc N. H. Bingham, C. M. Goldie, and J. L. Teugels }{\it Regular Variation} Cambridge University
Press, (1987).
\bibitem{bismut}{\sc Bismut J.-M.} {\it M\'ecanique A\'eatoire} Lect. N. in M. 866, Springer (1981).
\bibitem{bouleau4}{\sc Bouleau N.} "Error calculus and regularity of Poisson functionals: the lent particle method" {\it C. R. Acad. Sc. Paris,  Math\'ematiques},
Vol 346, n13-14, p779-782, (2008).

\bibitem{bouleau-hirsch2}{\sc Bouleau N.} and {\sc Hirsch F.} {\it Dirichlet Forms and Analysis on Wiener Space} De Gruyter (1991).
\bibitem{bouleau-denis} {\sc Bouleau N.} and {\sc Denis L.} ``Energy image density property and the lent particle method for Poisson measures" {\it Jour. of Functional Analysis} 257 1144-1174 (2009).
\bibitem{bouleau-denis2} {\sc Bouleau N.} and {\sc Denis L.} ``Application of the lent particle method to Poisson driven SDE's", to appear in Probability Theory and Related Fields.
\bibitem{coquio}{\sc Coquio A.} "Formes de Dirichlet sur l'espace canonique de Poisson et application aux \'equations diff\'erentielles stochastiques" {\it Ann. Inst. Henri Poincar\'e} vol 19, n1, 1-36, (1993).
\bibitem{denis2}{\sc  Denis L.}  "Quasi-sure Analysis for the Euler approximation and the flow related to an S.D.E.", in {\it  Dirichlet Forms and Stochastic Processes,} De Gruyter  p. 103-112  (1995).
\bibitem{ikeda-watanabe}{\sc Ikeda N., Watanabe S.} {\it Stochastic Differential Equation and Diffusion Processes}, North-Holland, Koshanda (1981).
\bibitem{jacod} {\sc J. Jacod} {\em Calcul Stochastique et Probl\`emes de Martingales}, Lect. Notes in Math. 714,
Springer, (1979).
\bibitem{kunita} {\sc Kunita H.} {\it Stochastic Flows and Stochastic Differential Equations}, Cambridge University Press (1990).
 \bibitem{latala}{\sc Latala R.}{"Estimates of moments and tails of Gaussian chaoses"} {\it The Annals of Probability} Vol. 34, No. 6, 2315-2331, (2006).
\bibitem{leandre1}{\sc L\'eandre R.} "R\'egularit\'e de processus de sauts d\'eg\'en\'er\'es (I), (II)" {\it Ann. Inst. Henri Poincar\'e} {\bf 21},  125-146, (1985); {\bf 24} , 209-236, (1988).
\bibitem{leandre2}{\sc L\'eandre R.} "Regularity of degenerated convolution semi-groups without use of the Poisson space" preprint Inst. Mittag-Leffler (2007).
\bibitem{ledoux-talagrand}{\sc Ledoux M., Talagrand M.} {\it Probability in Banach Spaces} Springer (1991).
\bibitem{meyer18}{\sc Meyer P.A.} "Transformations de Riesz pour les lois gaussiennens" {\it S\'em. Strasbourg XVIII} p179 Springer (1984).
\bibitem{picard}{\sc Picard J.}"On the existence of smooth densities for jump processes" {\it Probab. Theory Relat. Fields} 105, 481-511, (1996)
\bibitem{pisier}{\sc Pisier G.}"Les inégalités de Khintchine-Kahane, d'après C. Borell" {\it Séminaire Analyse fonctionnelle (dit "Maurey-Schwartz")}, exp. 7, (1977-1978). \\(available on \url{http://www.numdam.org/item?id=SAF_1977-1978____A6_0}) 
\bibitem{stein}{\sc Stein E. M.} "Some results in Harmonic Analysis in $\mathbb{R}^n$ for $n\rightarrow\infty$" {\it Bull. Amer. Math. Soc.} 9, 71-73, (1983).
\bibitem{wu1}{\sc Wu L.} "In\'egalit\'e de Sobolev sur l'espace de Poisson" S\'em. Proba. Strasbourg XXI p114 Springer (1987).


\end{thebibliography}
 \end{document}